%% file: main_full.tex
\documentclass{article}

\input packages.tex

\input definitions.tex

\numberwithin{equation}{section}

\newcommand{\lss}{\hat{\pi}_{\mathrm{LSS}}}
\newcommand{\lssc}{\hat{\pi}_{\mathrm{LSS-C}}}
\newcommand{\mle}{\hat{\pi}_{\mathrm{MLE}}}
\newcommand{\tpi}{\pi^{\star}}
\newcommand{\hpi}{\hat{\pi}}
\newcommand{\gdpi}{\hat{\pi}_{GD}}
\newcommand{\gaug}{G^{\text{aug}}}
\newcommand{\lssm}{\hamdist{\lss}{\tpi}}
\newcommand{\lsscm}{\hamdist{\lssc}{\tpi}}
\newcommand{\mlem}{\hamdist{\mle}{\tpi}}
\newcommand{\gm}{\hamdist{\hpi}{\tpi}}
\newcommand{\nnoise}{\tilde{Z}}
\newcommand{\noise}{Z}
\newcommand{\ipos}{X}
\newcommand{\dpos}{\Delta}
\newcommand{\idist}{\mathcal{P}}
\newcommand{\ndist}{\mathcal{Q}}

\newcommand{\bluenew}[1]{#1}

\colorlet{blue}{black}

\newcommand{\modlds}{\Cref{mod:low_dim}+S }
\newcommand{\bmodlds}{\textbf{\modlds}}

\begin{document}

\begin{titlepage}

\title{Geometric planted matchings beyond the Gaussian model}
{\color{blue}
\author[1]{Lucas R. Schwengber \thanks{lucas.schwengber@berkeley.edu, corresponding author. Work done as a student at IMPA.}}
\author[2]{Roberto I. Oliveira \thanks{rimfo@impa.br. Work supported by FAPERJ grant E-26/200.485/2023 (Cientista do Nosso Estado) and CNPq grant 305765/2023-0. The authors express no conflicts of interest.}}
\affil[1]{UC Berkeley, Department of Statistics, Berkeley, U.S.}
}
\bluenew{\affil[2]{IMPA, Rio de Janeiro, Brazil}}

\date{2026}

\maketitle

\subsection*{Abstract}

We consider the problem of recovering an unknown matching between a set of $n$ randomly placed points in $\R^d$ and random perturbations of these points. This can be seen as a model for particle tracking and more generally, entity resolution. We use matchings in random geometric graphs to derive minimax lower bounds for this problem that hold under great generality. Using these results we show that \bluenew{for a fixed $d$, as long as the noise distribution has finite $d$-th moment, and both initial positions and noise have bounded continuous densities, the minimax rate for the problem scales as $\Theta(n^2\sigma^d \wedge n)$. Under the stronger assumptions that the tail of the noise is sub-Gaussian, we show that} the order of the number of mistakes made by an estimator that minimizes the sum of squared Euclidean distances is minimax optimal when $d$ is fixed and is optimal up to $n^{o(1)}$ factors when $d = o(\log n)$. In the high-dimensional regime we consider a setup where both initial positions and perturbations have independent sub-Gaussian coordinates. In this setup we give sufficient conditions under which the same estimator makes no mistakes with high probability. We prove an analogous result for an adapted version of this estimator that incorporates information on the covariance matrix of the perturbations.

\end{titlepage}

\cleardoublepage

\section{Introduction}
\label{sec:intro}

We consider the geometric planted matching model {\color{blue} studied} by D. Kunisky and J. Niles-Weed in \cite{geo_match_2022}. Given $n, d \in \N$, we generate a sequence $X_1, \dots, X_n$ of independent and identically distributed (i.i.d.) random positions in $\R^d$ and independently a sequence of i.i.d. \hspace{-0.5em} random perturbations $\noise_1, \dots, \noise_n$, from which we derive,
\begin{equation}
\label{eq:final_pos}
    Y_i = X_{\tpi(i)} + Z_i, ~ i \in [n],
\end{equation} a sequence of perturbed versions of the original points with the label information hidden by an unknown random permutation $\tpi \sim \text{Unif}(S_n)$.
The goal is to recover $\tpi$ given direct observation of $X_1, \dots, X_n$ and $Y_1, \dots, Y_n$. 

We can interpret $\{X_1, \dots, X_n\}$ and $\{Y_1, \dots, Y_n\}$ as being snapshots at different times of the positions of $n$ particles performing some form of random motion in $\R^d$. Based on these snapshots we would like to track each particle's trajectory. This is a natural problem that arises in the life sciences \cite{particle_tracking_shen2017single, particle_tracking_chenouard2014objective,jaqaman2008robust} in the context of analyzing single-cell dynamics, and in the physical sciences \cite{la2001fluid} in problems involving fluid dynamics. This model also makes sense for growing $d$. For instance, if we imagine that the $X_i$'s are high-dimensional feature vectors, and the $Y_i$'s are noisy versions of the $X_i$'s, then estimating $\pi^*$ is a toy model for {\em entity resolution} or {\em deduplication} \cite{christen2012data}.

Under the assumption that $X_i \sim \mathcal{N}(0,I_d)$ and $Z_i \sim \mathcal{N}(0,\sigma^2 I_d)$, \cite{geo_match_2022} established sufficient and necessary conditions on $\sigma^2$ for the number of mistakes made by the maximum likelihood estimator (MLE) to be $0$ or $o(n)$ with high-probability for different regimes of $d$. Under the same Gaussian model, \cite{wang2022random} showed that some of these conditions are information-theoretical optimal. In \cite{collier2016minimax} and \cite{chen2022one} information-theoretical limits were also established under different models for database matching with Gaussian assumptions. Another line of research \cite{planted_match_exact_moharrami} considers the problem of recovering planted matchings in weighted bipartite graphs without a latent geometric structure. In that case, more general results are available \cite{planted_match_exact_moharrami, rec_threshold_sparse_semerjian, plant_sharp_thr_ding}.

In this work we consider the geometric planted matching problem following \eqref{eq:final_pos} for broader classes of distributions for $X_i$ and $Z_i$. In this more general setting, we obtain a general lower bound on the expected number of mistakes made by any estimator. Our proof is based on a simple argument involving matchings in random geometric graphs. The same strategy can also be used to establish negative results similar to the ones in \cite{geo_match_2022} and \cite{wang2022random}. We showcase some consequences for the regime when $d = \Theta(1)$ and $d = o(\log n)$ providing natural scenarios where it gives optimal results up to constant factors in the first regime, and up to $n^{o(1)}$ factors in the latter one.

We also obtain results in high-dimensional settings when both the initial positions and noise vectors have anisotropic sub-Gaussian distributions. This is motivated by the connection with data deduplication, where non-isotropy seems natural. Our main finding is that there seems to be natural ``effective dimension''~ and ``effective SNR''~ parameters that govern the matching problem in this setting, much like the true dimension determines the error rate for isotropic Gaussians \cite{geo_match_2022}.

{\color{blue}
Taken together our results indicate that the geometric planted matching problem exhibits a key difference relative to its non-geometric counterpart studied in \cite{plant_sharp_thr_ding,rec_threshold_sparse_semerjian}. Whereas in the non-geometric case, the MLE or likelihood-based variants are the default choices for positive results, in the geometric setting, there exist statistically reasonable estimators which always use the same geometric information about the initial and final positions, regardless of other specific details of the model. This key difference can only be seen when one goes beyond Gaussian assumptions, as for the Gaussian model the MLE coincides with an estimator that relies only on purely geometric information.
Our work makes progress on establishing under what general assumptions on the model, significant performance guarantees for such distribution-agnostic estimators can be established.
}

\subsection{Main results}
\label{sec:main_results}

In what follows, an {\em estimator} $\hat{\pi}$ is a map that takes as input two sequences of $n$ points in $\R^d$, $X_1, \dots, X_n$ and $Y_1, \dots, Y_n$, and outputs a permutation $\hat{\pi} \in S_n$. We measure the performance of such estimator $\hat{\pi}$ using the Hamming distance with respect to the true planted matching $\tpi$:
\begin{equation}
\label{eq:hamm_dist}
    \gm = \sum_{j=1}^n \indic{\hpi(j) \neq \tpi(j)}.
\end{equation}
As in previous works \cite{inference_particle_tracking, planted_match_exact_moharrami, rec_threshold_sparse_semerjian, plant_sharp_thr_ding, geo_match_2022, wang2022random} we want to understand the behaviour of the above quantity and its expectation for large $n$. More concretely, we would like to establish upper bounds on the expected number of mistakes made by a certain estimator (positive results), and lower bounds on the same quantity that hold for some, or for any, estimator (negative results).

\noindent For our positive results, we focus on an estimator that solves: 
\begin{equation}
\begin{aligned}
\label{eq:standard_mle}
    \lss &= \underset{\pi \in S_n}{\text{arg min}} \sum_{j=1}^n \ltwo{X_{\pi(j)}-Y_j}^2 \\
    &= \underset{\pi \in S_n}{\text{arg max}} \sum_{j=1}^n \langle X_{\pi(j)}, Y_j \rangle.
\end{aligned}
\end{equation}
{\color{blue}
Under the Gaussian model from \cite{geo_match_2022} this coincides with the MLE. We note that this fact only depends on the distribution of the noise, so one should expect that its behavior should not depend significantly on changes in the distribution of initial positions. It is less clear, however, how its behavior changes when using it with non-Gaussian noise. A main theme in the present work is investigating this question.}
Following \cite{collier2016minimax} we will refer to this estimator as the \textit{Least Sum of Squares Estimator} (LSS).

This estimator is distribution agnostic and can also be evaluated efficiently using the Hungarian method \cite{hungarian_method_kuhn} since it is an instance of a linear assignment problem. It appeared in previous theoretical analysis \cite{geo_match_2022, wang2022random, chen2022one, collier2016minimax} under Gaussian assumptions and in practical applications \cite{jaqaman2008robust}. Not directly related to this work, there exists also an extensive literature on understanding the asymptotic behaviour of so-called Euclidean Bipartite Matching Problems \cite{sicuro2017euclidean, talagrand2018matching}, which have the same structure as the random optimization problem in \eqref{eq:standard_mle}, but the $Y_i$'s are assumed to be an independent copy of the $X_i$'s.

\subsubsection{Low-dimension}
Our first two results are especially informative when $d = o(\log n)$. This includes the case where $d$ is fixed, which is the natural setting for the problem of particle tracking when the number of particles being tracked is very large \cite{inference_particle_tracking}.

\begin{model}{LD}
\label{mod:low_dim}
We generate a model following \eqref{eq:final_pos} with the assumption that:
\begin{enumerate}[(a)]
    \item \label{cond:i_pos_low_dim} $X_1, \dots, X_n \iid \idist$ where $\idist$ is a probability distribution on $\R^d$ such that:
    \begin{enumerate}[(i)]
        \item $\idist(\{x \in \R^d; ~ \norm{x} < R_d \}) \geq \gamma$ for some norm $\enorm$ and $R_d, \gamma > 0$.
        \item $\idist$ has a density $f_\idist$, with respect to Lebesgue measure in $\R^d$, such that:
        \begin{equation}
            \frac{\esssup\limits_{\norm{x} < R_d} f_\idist(x)}{\essinf\limits_{\norm{x} < R_d} f_\idist(x)} \leq e^{\beta d},
        \end{equation}
        for some constant $\beta \geq \log(2)$.
    \end{enumerate}
    \item \label{cond:noise_low_dim} $\noise_1, \dots, \noise_n = \sigma \nnoise_1, \dots, \sigma \nnoise_n$, where $\sigma > 0$ is the noise level and $\nnoise_1, \dots, \nnoise_n \iid \ndist$ are random perturbation directions. We assume $\ndist$ is a distribution in $\R^d$ such that:
   \begin{equation}
            \sup_{\norm{v} < R_d}\norm{\ndist(v)\otimes \ndist(-v) - \ndist \otimes \ndist}_{TV} \leq 1 - e^{-\beta d},
        \end{equation}
    where $\otimes$ denotes the product measure, $\ndist(v)$ is the law of $\tilde{Z}_1 + v$, and $\enorm, R_d$ and $\beta$ are the same as in assumption \ref{cond:i_pos_low_dim}.
\end{enumerate}
\end{model}

Intuitively speaking, assumption (a) says that there is a ball carrying a positive fraction of the mass of $\idist$ where the density $f_{\idist}$ does not vary too much. Assumption (b) is expected to be satisfied for instance when $\ndist$ has a density with a significant overlap with bounded translations of itself. For instance, these assumptions naturally hold for $\idist=\ndist=\mathcal{N}(0,I_d)$ with $\|\,\cdot\,\| = \ltwo{\, \cdot \,}$, $R_d=\sqrt{2d}$ and $\beta, \gamma>0$ which are independent of $d$ and $n$. But they also hold if we only assume that $\idist$ and $\ndist$ have bounded continuous densities (see \S \ref{subsec:fix_dim_app} below).

Under the above assumptions, we have following explicit minimax lower bound:
\begin{theorem}[Minimax lower bound (proof in \S \ref{sub:minimaxlowerbound})]
\label{thm:lb_low_d}
Under \Cref{mod:low_dim}, for all $n \geq 3$,
\begin{equation}
\label{eq:inf_lb}
       \inf_{\hpi} \Ex{\gm} \geq \frac{\gamma^2}{32} e^{-7\beta d} \left( (n^2 \sigma^d) \wedge n \right),
\end{equation}
where the infimum above is taken over all estimators $\hpi$.
\end{theorem}

The proof of this result is based on a natural argument that uses matchings in the random geometric graph $G(n, r, d, \enorm) = G(\{X_1, \dots, X_n\}, r, d, \enorm)$ constructed by joining pairs of initial positions which are within a distance $r$ of each other with respect to norm $\enorm$. As a main step in this proof we also show the following result of independent interest:
\begin{lemma}[Largest matching size lower bound (proof in \S \ref{sec:largest_matching_rgg})]
\label{lem:matching_exp}
    Let $X_1, \dots, X_n \iid \idist$, set $M_r$ to be the largest matching in $G(n, r, d, \enorm)$ and suppose $\idist$, $\enorm, R_d, \gamma$ and $\beta$ follow assumption \ref{cond:i_pos_low_dim} from \Cref{mod:low_dim}. Then for all $n \geq 3$ and $r > 0$,
   \begin{equation}
    \Ex{|M_r|} \geq \frac{\gamma^2}{16} e^{-6\beta d} \left( \left(n^2 \left( \frac{r}{R_d} \right)^d\right) \wedge n \right).
\end{equation}
\end{lemma}
Given \Cref{lem:matching_exp}, the idea of the proof of Theorem \ref{thm:lb_low_d} is that if two initial positions are sufficiently close to each other they are likely to be confused once noise is added. Therefore, given a suitable $r$, a large matching in $G(n,r,d,\enorm)$ implies a similarly large amount of mistakes when recovering $\hpi$. When $d$ is fixed, it suffices to consider $r = \Theta(\sigma)$, as illustrated in \Cref{fig:rgg_errors}. In this case the order $(n^2 \sigma^d) \wedge n$ appearing in \Cref{thm:lb_low_d} has a very natural geometrical interpretation: it is the order of the size of the largest matching of initial positions within distance $\Theta(\sigma)$.

\begin{figure}[H]
    \centering
    \includegraphics[scale=0.26]{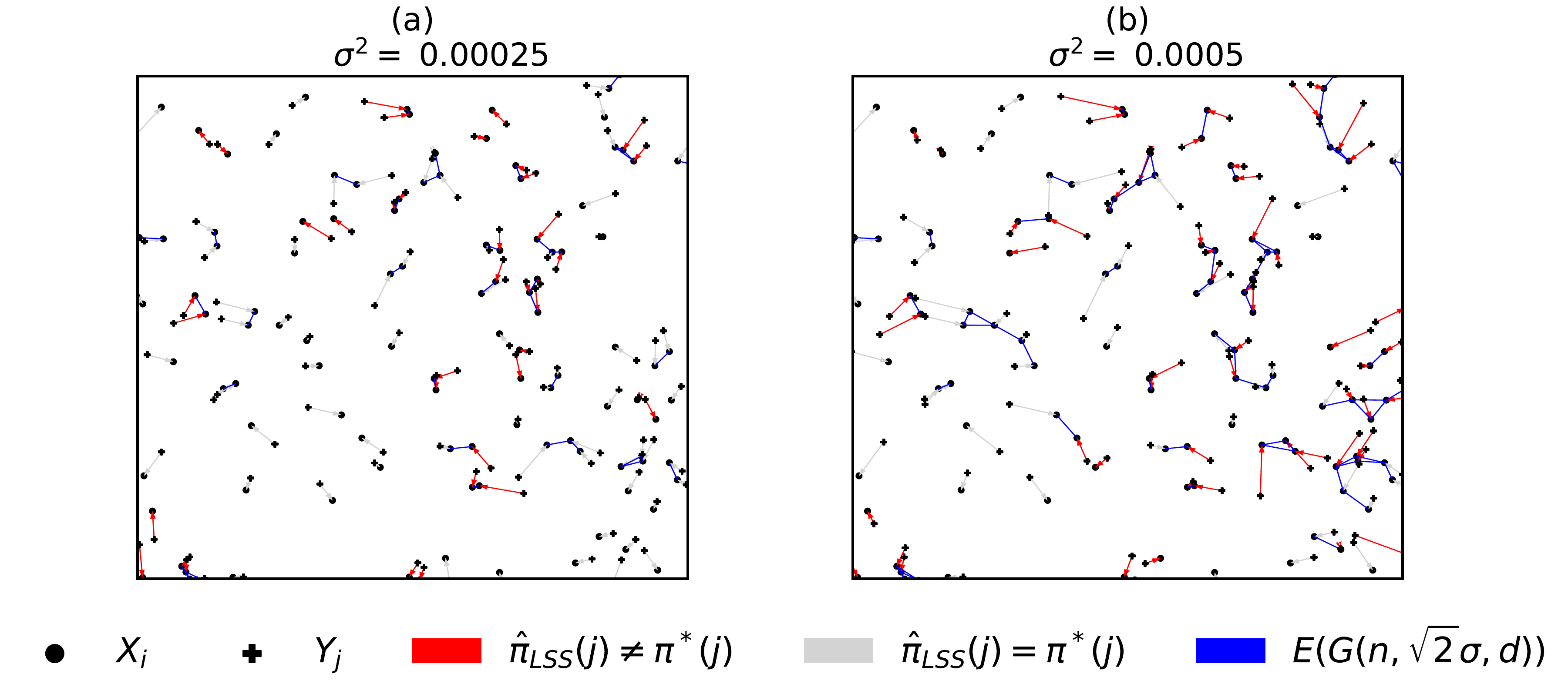}
    \caption{A small window of a simulation of the performance of $\lss$ on a sample following \eqref{eq:final_pos} of total size $n=3000$ with $X_1 \sim \mathcal{N}(0,I_d)$, $Z_1 \sim \mathcal{N}(0,\sigma^2 I_d)$ for two different values of $\sigma^2$. Arrows indicate to which initial position $X_i$ a given $Y_j$ was associated by $\lss$. Red arrows indicate that the estimated match for the given point was incorrect, while grey arrows indicate that the estimated match was correct. Blue lines between pairs of initial positions indicate that they are within a distance of $r = \sqrt{2}\sigma$ of each other. Note that most incorrect matches involve initial positions that share an edge on $G(n,r,d) = G(n,r,d,\ltwo{\,\cdot\,})$. Our strategy to prove our lower bounds takes inspiration from this observation. The code to generate this image can be found at \href{https://github.com/Lucas-Schwengber/particle_tracking}{https://github.com/Lucas-Schwengber/particle\_tracking}.}
    \label{fig:rgg_errors}
\end{figure}

Although to the best of our knowledge understanding the size of matchings in random geometric graphs has not received direct attention in previous works, \Cref{lem:matching_exp} is not significantly far from what is already well-known in the literature. For instance, if $d$ and $\idist$ are fixed one can use general results of the classical monograph by Penrose \cite[Propositions 3.2 and 3.3]{rggpenrose} to count isolated edges in a random geometric graph, giving a lower bound on the size of its largest matching which is sharp up to constant factors. For our purposes, the main advantage of Lemma \ref{lem:matching_exp} is that it is general, nonasymptotic and can be made to work in growing dimensions. Moreover, its proof is simple and constructive.

\begin{remark}
Our approach to establish lower bounds is different from what was done in \cite{geo_match_2022} where the authors look at matchings in a different random graph, which is built by adding edges between the endpoints of augmenting $2$-cycles (see \Cref{sec:lss_lb} for a definition). There is a direct connection between the two approaches, which we outline in \Cref{a:high_prob_lb}. In particular, when $d = o(\log n)$ or $d$ is fixed, our techniques are able to recover the lower bounds from \cite{geo_match_2022} up to factors of order $n^{o(1)}$, while working under greater generality (see \Cref{thm:hp_lss_lb}).
\end{remark}

\begin{remark} The proof strategy of our negative results also differs from the one in \cite{wang2022random} where the authors either count augmenting $2$-cycles as in \cite{geo_match_2022} or use a mutual information argument. But one can also prove a version of \Cref{thm:lb_low_d} that gives a lower bound in probability for the performance of any estimator, thus allowing us to derive negative results that complement and extend the negative results from \cite{wang2022random} (see \Cref{cor:hp_lbs} in \Cref{a:high_prob_lb}) when $d$ is fixed and when $d = o(\log n)$. 
\end{remark}

\begin{remark}
    We note that, similar to this work, \cite{chen2022one} and \cite{collier2016minimax} also established information-theoretical limits using the distances between uncorrupted data points. However their results are stated in terms of the \textit{minimum distance} between such points. An interesting feature of these results is that they hold \textit{conditionally} on the set of uncorrupted data points. As a main step in our proof of \Cref{thm:lb_low_d} we also derive a result that holds conditionally on the set of original positions $X_i$ (see \Cref{lem:gen_inf_theo_lb}). However our results are in terms of $|M_r|$, the size of the largest matching in $G(n,r,d,\enorm)$ as defined previously.
\end{remark}

\begin{remark}
    Although the negative results established in this paper concern the case when $d = o(\log n)$, we believe our strategy can also lead to non-trivial results when $d = \Omega(\log n)$. In particular, in \Cref{a:log_reg} we showcase how they might be used to partially recover the lower bounds from \cite{geo_match_2022} in the regime where $d = \Theta(\log n)$. 
\end{remark}

Our next result shows that the LSS estimator achieves minimax optimal  rates in a slightly more restrictive setting than \Cref{mod:low_dim}. Specifically, let \bmodlds be \Cref{mod:low_dim} with the additional assumptions that:
\begin{enumerate}[(a)]
\setcounter{enumi}{2}
    \item\label{cond:up_density} $\linf{f_\idist} \leq e^{\beta d}$, for the same $\beta$ as in assumption \ref{cond:i_pos_low_dim}.
    \item\label{cond:sub_gauss} The distribution of the $\nnoise_i$'s is centered and satisfies:
    \begin{equation}
        \snorm{\nnoise_i} \leq K_Q,
    \end{equation}
    where $\snorm{\, \cdot \,}$ denotes the sub-Gaussian norm.
\end{enumerate}
Under these additional assumptions, we have the following upper bound on $\Ex{\lssm}$:
\begin{theorem}[LSS error upper bound (proof in \S \ref{sec:up_bound_ld})]
\label{thm:lss_up_low_d}
Under \modlds\hspace{-0.75em} , for all $n \geq 2$,
\begin{equation}
     \Ex{\lssm} \leq (3K^d n^2 \sigma^d) \wedge n,
\end{equation}
where $K = b_{\mathrm{LD}} K_\ndist^2 e^{2\beta}$ for some absolute constant $b_{\mathrm{LD}} > 0$.
\end{theorem}
In \Cref{subsec:fix_dim_app} below we show that if $d$ is fixed we can improve our constant $3K^d$ to a constant $\tau_{d, \idist, \ndist}$ which we conjecture to be sharp when $n^2 \sigma^d \ll n$ (see \Cref{fig:error_rate}). This constant depends very mildly on the geometrical properties of $\idist$ and $\ndist$ thus suggesting that when $n^2 \sigma^d = o(n)$, $\Ex{\lssm}$ behaves as in the Gaussian model for a broad class of distributions.

The proof of \Cref{thm:lss_up_low_d} follows a similar strategy as in \cite{geo_match_2022}. Given a $t$-cycle $C_t = (i_1,i_2,\ldots,i_t)$ of distinct elements from $[n]$ we say that it is an \textit{augmenting cycle} (for $\lss$), or just \textit{augmenting}, if:
\begin{equation}
    \sum_{k=1}^t \langle X_{\pi^*(i_k)}, Y_{i_{k+1}} \rangle \geq \sum_{k=1}^t \langle X_{\pi^*(i_k)}, Y_{i_k} \rangle,
\end{equation}
where $i_{t+1} := i_1$. If we assume that $\tpi = \text{Id}$, all cycles of size $2$ or greater on the cycle decomposition of $\lss$ will be mistakes. Moreover, these cycles must be \textit{augmenting cycles}, otherwise they would not appear in the optimal assignment $\lss$. An analogous reasoning also applies when $\tpi \neq \mathrm{Id}$. It follows that we have the following deterministic bound:

\begin{equation}
\label{eq:det_bound_lss}
    \lssm \leq \sum_{t=2}^n \sum_{C_t} t \indic{C_t \text{ is augmenting}},
\end{equation}
 where the inner sum is over all $t$-cycles $C_t$. 

To bound $\Ex{\lssm}$, we  bound the probability that $t$-cycles are augmenting. Similarly to \cite{geo_match_2022} our upper bounds on these probabilities involve bounding the expectation of exponentials of certain quadratic forms on the initial positions. Unlike in the Gaussian setting of \cite{geo_match_2022}, we cannot evaluate these expectations explicitly. Fortunately, the assumptions in \modlds are good enough for calculations to go through.

\bluenew{We note that just like in \cite{geo_match_2022}, the sharpness of our upper bounds, when they apply, rely crucially on the fact that the mass of augmenting $2$-cycles dominate the error of $\lss$ up until the number of errors is linear. We believe outside this regime, understanding the $\lss$ is significantly more challenging. As in \cite{geo_match_2022}, our experiments (see \Cref{fig:error_rate}) suggests that other geometrical aspects of the problem beyond the mass of augmenting $2$-cycles play a role once the regime of linear mistakes starts.}

In \Cref{sec:examples} we show that Theorems \ref{thm:lb_low_d} and \ref{thm:lss_up_low_d} can be combined to obtain optimal, and nearly-optimal, rates for some natural choices of $\idist$ and $\ndist$. One such rate is the Gaussian model studied in \cite{geo_match_2022} in the regime where $d = o(\log n)$, where our bounds are optimal up to $n^{o(1)}$ factors. Another such case is when $\idist$ and $\ndist$ have continuous and bounded densities with $d$ fixed. We additionally consider the cases of $\idist$ uniform or from the Laplace family. 

\subsubsection{High-dimension}

Our next results concern the high-dimensional setting. As mentioned above, this is a natural toy model for \textit{deduplication} and  \textit{entity resolution} between high-dimensional datasets \cite{christen2012data}.

\begin{model}{HD}
\label{mod:noniso_subgauss}
In this setup we let $n \to +\infty$ while making the following assumptions in \eqref{eq:final_pos}:
\begin{enumerate}[(a)]
    \item $d = \omega(\log n)$.
    \item $X_1, \dots, X_n = (\Sigma_X)^{\frac{1}{2}} \tilde{X}_1, \dots, (\Sigma_X)^{\frac{1}{2}} \tilde{X}_n$ and $Z_1, \dots, Z_n = (\Sigma_Z)^{\frac{1}{2}} \nnoise_1, \dots, (\Sigma_Z)^{\frac{1}{2}} \nnoise_n$ where:
    \begin{enumerate}[(i)]
        \item $\tilde{X}_1, \dots, \tilde{X}_n$ are i.i.d. and independent from $\nnoise_1, \dots, \nnoise_n$, which are also i.i.d.
        \item \label{itm:kz} Both $\tilde{X}_1$ and $\nnoise_1$ have centered, unit variance and independent coordinates with:
        \begin{equation}
        \max_{j \in [d]}\snorm{\tilde{X}_1[j]} \leq K_X\text{, } \ \max_{j \in [d]}\snorm{\nnoise_1[j]} \leq K_Z,
        \end{equation}
        for constants $K_X, K_Z > 0$ that do not depend on $n$.
        \item $\Sigma_X = \diago{\sigma_X[1], \dots, \sigma_X[d]}$ and $\Sigma_Z = \diago{\sigma_Z[1], \dots, \sigma_Z[d]}$ are diagonal matrices with strictly positive diagonal entries.
    \end{enumerate}
\end{enumerate}
\end{model}
We leave implicit the dependence on $n$ of the distribution of the $X_i$'s, $Z_i$'s and $Y_i$'s as well as $\Sigma_X$ and $\Sigma_Z$.
Under the above model we have the following result:
\begin{theorem}[Perfect recovery for LSS in high-dimensions (proof in \S \ref{sec:up_bound_hd})]
\label{thm:high_dim_lss} Under \Cref{mod:noniso_subgauss}, suppose $\frac{\tr{\Sigma_Z \Sigma_X}}{\opnorm{ \Sigma_Z \Sigma_X }} = \omega(\log n)$, $\frac{\tr{\Sigma_X}}{\opnorm{\Sigma_X}} = \omega(\log n)$ and let,
\begin{equation}
 \tilde{\sigma}^2_X[j] = \frac{\sigma^2_X[j]}{\sum\limits_{k=1}^d \sigma^2_X[k]}, ~ j \in [d].   
\end{equation}
There exists an absolute constant $b_{\mathrm{HD}} > 0$ such that if
    \begin{equation}
    \label{eq:lss_upper_bound}
        b_{\mathrm{HD}} K^2_Z \leq \log n^{-1} \frac{\sum\limits_{j=1}^d \sigma^2_X[j]}{\left(\sum\limits_{j=1}^d \tilde{\sigma}^2_X[j]\sigma^2_Z[j]\right)}
    \end{equation}
    then $\lssm = 0$ with high probability.  
\end{theorem}
This theorem establishes a sufficient condition for perfect recovery for an estimator which is agnostic to any information on the distribution of the noise and its covariance structure. It extends the results on the high-dimensional setting from \cite{geo_match_2022}. This condition can be slightly improved if we assume knowledge of $\Sigma_Z$. We consider a modified version of LSS, LSS-C, which incorporates the information from $\Sigma_Z$. It is defined as:

\begin{equation}
\begin{aligned}
\label{eq:lss_c}
    \lssc &:= \underset{\pi \in S_n}{\text{arg min}} \sum_{j=1}^n \ltwo{(\Sigma_Z)^{-\frac{1}{2}}(X_{\pi(j)}-Y_j)}^2 \\ 
    &= \underset{\pi \in S_n}{\text{arg max}} \sum_{j=1}^n \left\langle (\Sigma_Z)^{-1}X_{\pi(j)}, Y_j \right\rangle.
\end{aligned}
\end{equation}

\bluenew{Much in the same way that $\lss$ can be seen as the MLE under isotropic Gaussian noise, $\lssc$ can be seen as the MLE under Gaussian noise with a general covariance structure given by $\Sigma_Z$. Equivalently, the $\lssc$ may also be seen as $\lss$ applied to a transformation $\Sigma_Z^{-\frac{1}{2}}$ of the underlying space which whitens out the covariance structure of the noise.}

Regarding this estimator we have the following sufficient condition for perfect recovery:
\begin{theorem}[Perfect recovery for LSS-C in high-dimensions (proof in \S \ref{sec:up_bound_hd})]
\label{thm:high_dim_lssc}
Under \Cref{mod:noniso_subgauss}, suppose $\frac{\tr{(\Sigma_Z)^{-1}\Sigma_X}}{\opnorm{ (\Sigma_Z)^{-1}\Sigma_X }} = \omega(\log n)$ and let $(\tilde{\sigma}^2_X[j])_{j=1}^d$ and $b_{\mathrm{HD}}$ be as in \Cref{thm:high_dim_lss}. If
    \begin{equation}
    \label{eq:lssc_upper_bound}
    \begin{split}
        b_{\text{HD}} K^2_Z 
        &\leq \log n^{-1} \frac{\sum\limits_{j=1}^d \sigma^2_X[j]}{\left( \sum\limits_{j=1}^d \tilde{\sigma}^2_X[j]  \sigma^{-2}_Z[j] \right)^{-1}}, 
    \end{split}
    \end{equation}
    then $\hamdist{\lssc}{\tpi} = 0$ with high probability.
\end{theorem}

\bluenew{\vspace{1em}}

\bluenew{We expect the terms on the right hand side of \Cref{eq:lss_upper_bound} and \Cref{eq:lssc_upper_bound} to play the role of signal-to-noise ratios beyond the isotropic Gaussian assumptions. By this we mean that complementary negative results may be established for the corresponding estimators, when the (conjectured) signal-to-noise ratios are below a certain threshold.}

{\color{blue}
For instance, when $X_1,\dots,X_n \iid \mathcal{N}(0,I_d)$ and  $\noise_1, \dots, \noise_n \iid \mathcal{N}(0,\sigma^2 I_d)$, the authors in \cite{geo_match_2022} conjecture that there exists a phase transition as $n \to +\infty$ from \textit{partial recovery} (\ie $\Ex{\lssm} = \Omega(n)$) when $\frac{d}{\log n\sigma^2} < 4$ to \textit{perfect recovery} (\ie $\Ex{\lssm} = o(1)$) when 
$\frac{d}{\log n\sigma^2} > 4$. Similar to our results, they prove the sufficient condition for perfect recovery in terms of $\frac{d}{\log(n)\sigma^2}$. In \cite{wang2022random} Appendix E.1 the authors show that when $d = \omega(\log n)$, any estimator that achieves \textit{strong recovery} (\ie $\Ex{\lssm} = o(n)$) must have $\frac{d}{\log n\sigma^2} > 2$. This shows that, in the informal sense we defined, $\frac{d}{\log n \sigma^2}$ acts as signal-to-noise ratio for the $\lss$ estimator. We remark that determinig whether the sharp constant $4$ is correct for the $\lss$ seems to still be an open problem, to the best of our knowledge. 
}

\bluenew{Under the aforementioned assumptions, the right hand side of both \Cref{thm:high_dim_lss} and \Cref{thm:high_dim_lssc} simplify to the same $\frac{d}{\log n\sigma^2}$, recovering the signal-to-noise ratio for the isotropic Gaussian model, however allowing for making conjectures on the signal-to-noise ratio in more general cases. We, however, do not see our techniques as suitable for yielding conjectures regarding the \textit{exact} constants on which the phase transitions happen, as they hinge on techniques from high-dimensional probability which are usually agnostic to absolute constants \cite{vershynin_2018, rudelson2013hanson}.}

{\color{blue}
We note however, that under non-isotropic Gaussian assumptions, if the right hand side of \Cref{eq:lss_upper_bound} and \Cref{eq:lssc_upper_bound} both act as signal-to-noise ratios for the corresponding estimators, there must be a regime in which $\lssc$ is able to achieve \textit{perfect recovery}, while $\lss$ only achieves \textit{partial recovery}, so long as,
\begin{equation}
    \label{eq:lssc_beats_lss}
    \left( \sum\limits_{j=1}^d \tilde{\sigma}^2_X[j]  \sigma^{-2}_Z[j] \right)^{-1} \ll \left(\sum\limits_{j=1}^d \tilde{\sigma}^2_X[j]\sigma^2_Z[j]\right).
\end{equation}

This would provide a concrete case in which $\lss$ fails, but $\mle$ succeeds.} As this also suggests, the difference between \Cref{thm:high_dim_lss} and \Cref{thm:high_dim_lssc} lies in how the variance of the coordinates of the initial positions and noise interact. In \Cref{thm:high_dim_lss} the noise term includes a weighted arithmetic mean of the variances of the noise vector with weights given by the proportion of each coordinate on the signal term. The noise term in \Cref{thm:high_dim_lssc} has a similar interpretation but with a weighted harmonic mean. Since the harmonic mean is always smaller than the arithmetic mean, the (conjectured) signal-to-noise ratio term in \Cref{thm:high_dim_lssc} always dominates the one from \Cref{thm:high_dim_lss}, as expected.
The conditions,
\begin{equation}
\label{eq:stable_ranks}
    \frac{\tr{\Sigma_Z \Sigma_X}}{\opnorm{ \Sigma_Z \Sigma_X }} = \omega(\log n),\, \frac{\tr{\Sigma_X}}{\opnorm{\Sigma_X}} = \omega(\log n) \text{ and }  \frac{\tr{(\Sigma_Z)^{-1}\Sigma_X}}{\opnorm{ (\Sigma_Z)^{-1}\Sigma_X }} = \omega(\log n) 
\end{equation}
are technical conditions on the stable rank of $\Sigma_Z \Sigma_X$, $\Sigma_X$ and $(\Sigma_Z)^{-1}\Sigma_X$ respectively. These conditions roughly state that the sum of the entries of these diagonal matrices cannot be too concentrated along a small subset of coordinates. All three conditions necessarily require that $d = \omega(\log n)$. For $\Sigma_X$ this condition can be interpreted as saying that we must \textit{effectively} be in a high-dimensional setting by ruling out scenarios where we might have $d = \omega(\log n)$ but our signal is concentrated in a low-dimensional subspace. If $\Sigma_X$ is isotropic, the first condition in \eqref{eq:stable_ranks} rules out the possibility of the strength of noise being concentrated along a small subset of coordinates and the last condition rules out an analogous possibility regarding the inverse of the noise level on the coordinates.

Before proceeding, we note that the proofs of our high-dimensional results also rely on the concept of augmenting cycles that was outlined in the discussion of \Cref{thm:lss_up_low_d}. The main difference is that we must use the Hanson-Wright inequality to show that the quadratic forms on the initial positions concentrate around their expected value. This allows us to obtain sharp estimates as the dimension grows. 

{\color{blue}
\subsection{Comparison with other estimators}
\label{subsec:other_est}
}

\bluenew{\subsubsection{Maximum likelihood estimator}}

{\color{blue}
Unlike in the Gaussian model, for general distributions for the noise, $\lss$ does not coincide with the maximum likelihood estimator (MLE). Assuming $\mathcal{Q}$ is absolutely continuous with density $q$, we have that the MLE is given by:
\begin{equation}
    \mle := \underset{\pi \in S_n}{\text{arg max}} \sum_{j=1}^n  W^{\mathcal{Q},\sigma}_{\pi(j),j}
\end{equation}
}

{\color{blue}
where $W^{\mathcal{Q},\sigma}_{i,j} := \log\left(q\left( \frac{X_i - Y_j}{\sigma} \right) \right)$. The first thing to note is that this estimator does not depend on knowledge of $\mathcal{P}$. Moreover, it is also a random assignment problem which can be efficiently solved with the Hungarian method. The crucial difference, is that unlike in \cite{plant_sharp_thr_ding}, the weights $W_{i,j}^{\mathcal{Q}, \sigma}$ are not independent.
This prevents the establishment of general upper bounds on its errors as in \cite{plant_sharp_thr_ding}. In this work, we get around this issue by working with $\lss$ instead, which allows us to handle correlation through the underlying geometry. As discussed earlier in the introduction, this points to an asymmetry between planted matching recovery in geometrical vs. non-geometrical settings. In the geometrical setting we have candidates for good estimators which require little knowledge of the distribution of the noise, unlike in the non-geometric planted matching case in which positive results use likelihood-based estimators. 
}

{\color{blue}
Although we believe a general theory for $\mle$ might be possible, we believe it to be very challenging due to the correlations introduced. Nevertheless, our main conjecture concerning the comparison between $\mle$ and $\lss$ is as follows:
\begin{conjecture}
    \label{conj:mle_beats_lss}
    For all regimes in which the MLE makes sense, we have that $\Ex{\mlem} \leq \Ex{\lssm}$. 
\end{conjecture}
}

\bluenew{The reason why we believe this to be true, is that $W^{\mathcal{Q},\sigma}_{i,j}$ strictly incorporates more information about the noise distribution than $\ltwo{X_i-Y_j}^2$. For example, in \Cref{fig:error_rate}, which we discuss below, if the set of initial positions is absolutely continuous, the MLE would allow us to obtain perfect recovery a.s. for Spherical and Rademacher distributions, as we obtain (informally) that a.s., $W^{\mathcal{Q}, \sigma}_{i,j} = -\infty$ for all pairs with $i \neq \pi(j)$, since they will be a.s. at a distance $\neq \sigma \sqrt{d}$ from $Y_{j}$. This latter example shows how significant the difference between both can be. Whereas $\mle$ makes no mistakes regardless of $\sigma$, $\lss$ makes an increasingly larger number of mistakes.}

\bluenew{Moreover, the shortcomings that the $\mle$ might have in terms of the existence of large augmenting cycles, are also shared with $\lss$ (unlike the greedy algorithm we discuss below).}

{\color{blue}
A slightly more rigorous motivating heuristic for \Cref{conj:mle_beats_lss} in certain cases can be derived as follows. If the mass of augmenting $2$-cycles dominates the error of both $\Ex{\mlem}$ and $\Ex{\lssm}$, then we may establish \Cref{conj:mle_beats_lss} by using the following result (whose proof we defer to \Cref{a:other_estim}):
\begin{lemma}
    \label{lem:mle_beats_lss_2_cyc}
    Assuming $\tpi = \mathrm{Id}$ and $\mathcal{Q}$ is absolutely continuous, we have,
    \begin{equation}
        \begin{split}
            \Pr{W^{\mathcal{Q},\sigma}_{1,2} + W^{\mathcal{Q},\sigma}_{2,1} > W^{\mathcal{Q},\sigma}_{1,1} + W^{\mathcal{Q},\sigma}_{2,2}} \leq \Pr{\ltwo{X_1-Y_2}^2 + \ltwo{X_2-Y_1}^2 < \ltwo{X_1-Y_1}^2 + \ltwo{X_2-Y_2}^2}.
        \end{split}
    \end{equation}
    In other words, the probability of $(1 ~ 2)$ being an augmenting cycle for $\mle$ is never larger than that of $(1 ~ 2)$ being an augmenting cycle for $\lss$.
\end{lemma}
}

{\color{blue}
This result implies that the mass of augmenting $2$-cycles of the $\mle$ is always no greater than that of $\lss$. This implies that once the errors are dominated by this quantity, \Cref{conj:mle_beats_lss} should hold. An interesting direction for future work is stating general assumptions under which the MLE is expected to perform significantly better. We believe some sufficient conditions are: the tail being heavier than Gaussian (for example the Laplace noise discussed in \Cref{sub:log_lip_dens}), the noise being non-isotropic in meaningful ways as in \eqref{eq:lssc_beats_lss}, or the noise having some very rigid structure like in the Spherical and Rademacher examples discussed above. 
\subsubsection{Greedy distance}
}

\bluenew{Following \cite{geo_match_2022}, another algorithm we might consider is the so-called \textit{greedy distance} algorithm. This algorithm matches each $Y_j$ with the closest $X_i$ (according to $\ltwo{~ \cdot ~}$) if it is available, and declares an error in case it is not available. As we establish on \Cref{a:other_estim}, we expect this estimator to be nearly tied with the $\lss$ for the low-dimensional regime, under sub-Gaussian tails. We expect it to fare worse than the MLE in high-dimensions with sub-Gaussian noise, following on the results from \cite{geo_match_2022}.}

\bluenew{We note however, that $\gdpi$ has an error rate which is easier to bound without directly assuming sub-Gaussian tails. This is because there is no need to control the probability of augmenting cycles of arbitrary size. We specifically exploit this robustness to heavier tails in the result for fixed $d$ established in \Cref{prop:minimax_fix_d}. We, however, do not expect these gains to provide significantly better results as $d$ grows, even when $d=o(\log(n))$. As discussed in \Cref{a:other_estim}, this is because obtaining nearly sharp bounds for $d = o(\log(n))$ with either $\lss$ or $\gdpi$ with our present techniques requires $\Ex{\ltwo{\tilde{Z}_1}^d} = O(d^{\frac{d}{2}})$, which is as strong as requiring sub-Gaussian tails \cite{vershynin_2018}. 
We believe however, that this tail dependence is predicated on the choice of the norm $\|\cdot\|_2$ and hence we expect the MLE, or estimators using different distances, to be able to give improved rates, for instance in the model discussed in \Cref{sub:log_lip_dens}. 

}

\subsection{Organization}

The remainder of the paper is organized as follows. In \Cref{sec:notation} we set our notation. \Cref{sec:examples} presents examples that showcase some consequences of our results when $d = o(\log n)$. In \Cref{sec:lower_bounds_rggs} we discuss matchings in random geometric graphs and their application to minimax lower bounds. In particular, this is where we prove \Cref{thm:lb_low_d} and \Cref{lem:matching_exp}. \Cref{sec:up_bounds} discusses augmenting cycles and proves our upper bound results in both low and high-dimensions. \Cref{sec:conj_future_work} provides some directions for future work. The Appendix contains comments on and extensions of the arguments in the main text.

\subsection{Notation}
\label{sec:notation}

Throughout this work we will use the following conventions: \\

\noindent\textbf{General notation:} $\N, \Z, \R$ denote the set of positive integers, integers and real numbers respectively and for any $m \in \N$, $[m] = \{1, \dots, m\}$. \\

\noindent \textbf{Linear Algebra and Analysis:}  For a vector $v \in \R^d$ we denote its $i$-th coordinate by $v[i]$ and its Euclidean norm by $\ltwo{v}$. $\enorm$ denotes an arbitrary norm on $\R^d$. If $A$ is a $m \times n$ matrix, we denote by $A[i,j]$ the entry at its $i$-th row and $j$-th column. $A[i,:] \in \R^n$ is the $i$-th row and $A[:,j] \in \R^m$ is the $j$th column of $A$ represented as vectors. If $v \in \R^n$, $Av \in \R^m$ denotes the matrix-vector product. 

\begin{equation}
    \opnorm{ A } = \sup_{\|x\| = 1} \ltwo{Ax} \text{ and } \tr{A} = \sum\limits_{k=1}^n A[k,k]
\end{equation}
are the operator norm and trace of $A$ respectively, while $\fnorm{ A }$ is the Frobenius norm of $A$. If $B$ is another $k \times l$ matrix, $A \otimes B$ denotes the Kronecker product between $A$ and $B$. If $k=n$ then, $AB$ denotes the matrix product of $A$ and $B$, and if $k=m$ and $l=n$, $\binner{A}{B} = \tr{B^T A}$ is the inner product between $A$ and $B$.  $B_{\enorm}(v,s) = \{ x \in \R^d; ~ \norm{x-v} < s \}$ is the open ball centered at $v$ with radius $s > 0$ on norm $\enorm$. $\rho_d$ denotes the volume of $B_{\|\,\cdot\,\|_2}(0,1)$ in dimension $d$. $\wedge, \vee$ denote the minimum and maximum between two real numbers. $\log(x)$ denotes the natural logarithm of $x > 0$. \\

\noindent \textbf{Discrete structures:} $\hamdist{\omega}{\omega'}$ denotes the Hamming distance between two strings $\omega$ and $\omega'$ of equal length. If $G = (V,E)$ is a graph then $V(G)$ and $E(G)$ denote its set of vertices and edges respectively. $\binom{n}{k}$ is the binomial coefficient. If $A$ is a discrete set we denote its cardinality by $|A|$. $S_n$ is the group of permutations on $[n]$. \\

\noindent \textbf{Asymptotic Notation:} Let $a_n = \rvseq{a}$ and $b_n = \rvseq{b}$ be two sequences of real numbers. The asymptotic notation $O(\cdot), o(\cdot), \omega(\cdot)$, $\Omega(\cdot), \Theta(\cdot)$ has its standard meaning, $a_n \ll b_n$ and $ a_n \gg b_n$ are synonymous with $a_n = o(b_n)$ and $b_n = o(a_n)$ respectively. $a_n \sim b_n$ indicates that $\frac{a_n}{b_n} \tolim{n} 1$. \\

\noindent \textbf{Probability:} $\Pr{\cdot}$ and $\Ex{\cdot}$ denote probabilities and expectations respectively. We say $Z \sim \ndist$ if the random element $Z$ has law $\ndist$ and $Z_1, \dots, Z_n \iid \ndist$ indicates a sequence of independent and identically distributed (i.i.d.) random elements with common law $\ndist$. $\mathcal{N}(\mu,\Sigma)$ denotes a multivariate Gaussian distribution with mean $\mu \in \R^d$ and $d \times d$ covariance matrix $\Sigma$. $\text{Unif}(S_n)$ is the uniform distribution over $S_n$. If $Z \sim \mathcal{N}(0,1)$ then $\Phi(x) = \Pr{Z \leq x}$ is the cumulative distribution function of a standard Gaussian random variable $Z$. If $\ndist_1$ and $\ndist_2$ are two probability distributions on $\mathcal{Z}_1$ and $\mathcal{Z}_2$, $\ndist_1 \otimes \ndist_2$ denotes the product measure on $\mathcal{Z}_1 \times \mathcal{Z}_2$ with marginals $\ndist_1$ and $\ndist_2$. When $\mathcal{Z}_1 = \mathcal{Z}_2$, $\tvnorm{\ndist_1 - \ndist_2}$ indicates the total variation distance between $\ndist_1$ and $\ndist_2$. 

We say that a sequence of events $A_n = \rvseq{A}$ happens with high probability if $\Pr{A_n} = 1 - o(1)$. Given a sequence $U_n = \rvseq{U}$ of random variables, we say that $U_n = \Theta(a_n)$ with high probability if there exists absolute constants $c$ and $C$, not depending on $n$, such that $c a_n \leq U_n \leq C a_n$ holds with high probability. 
\begin{equation}
    U_n \pc{n} U, U_n \asc{n} U \text{ and } U_n \dc{n} U
\end{equation}
denote convergence in probability, almost sure convergence and convergence in law respectively. If $U'$ is another random variable, then $U \stackrel{d}{=} U'$ indicates equality in law. $o_p(1)$ denotes a random quantity that converges to $0$ in probability as $n \to +\infty$. $\snorm{U}$ denotes the sub-Gaussian norm of $U$ defined as:
{\color{blue}
\begin{equation}
    \snorm{ U } := \inf \set{L > 0; ~ \Ex{\expo{\frac{U^2}{L^2}}} \leq 2 }.
\end{equation}
If $V$ is a random vector, $\snorm{V} = \sup_{\norm{v}=1} \snorm{\binner{V}{v}}$ is the sub-Gaussian norm of $V$.
}

\subsection*{Note on LLM usage}

Some aspects of the revision of this paper were done aided by Gemini and Claude CLI's. The models were used to create latex macros and double check the text for typos and small inconsistencies on literature review, mathematical proofs and language use. All edits suggested by the language models were either carried out by the authors, or manually approved them. No edit suggested by an LLM involved any substantial intellectual revisions of the text, but merely small corrections for typos (e.g. wrong sign in a proof with no major change to the proof correctness), small aspects of language usage (e.g. asking for revision of overly colloquial language) and minor aspects of references to the literature. 

\section{Examples in the low-dimensional regime}
\label{sec:examples}

In this section we show that \Cref{thm:lb_low_d} and \Cref{thm:lss_up_low_d} can be used to derive positive and negative results for multiple choices of distributions for the $X_i$'s and $Z_i$'s.

\subsection{Gaussian initial positions and noise}
\label{sub:gaussian_low_dim}
We first consider the Gaussian model from \cite{geo_match_2022} and show that the rate obtained for the MLE is informational theoretical optimal up to factors of order $n^{o(1)}$. More precisely we have:

\begin{proposition}[Low-dimension, Gaussian initial positions, Gaussian noise] In \eqref{eq:final_pos}, let $d = o(\log n)$, $X_1 \sim \mathcal{N}(0,I_d)$ and $Z_1 \sim \mathcal{N}(0,\sigma^2 I_d)$. Then as $n \to +\infty$,
\begin{equation}
    \inf_{\hpi} \Ex{\gm} = (n^{(2 + o(1))} \sigma^d) \wedge n^{(1 + o(1))}.
\end{equation}    
\end{proposition}
\begin{proof}
If we take $\idist = \mathcal{N}(0,I_d)$ and $R_d = \sqrt{2d}$, we have that for all sufficiently large $d$, assumption \ref{cond:i_pos_low_dim} from \Cref{mod:low_dim} holds with $\enorm = \|\,\cdot\,\|_2$, $\gamma = \frac{1}{2}$ and $\beta = 1$. If we let $\ndist = \mathcal{N}(0,I_d)$, using Bretagnolle–Huber inequality \cite{brethuber} and the formula for the Kullback-Leibler divergence between multivariate Gaussian distributions we have:
\begin{equation}
\begin{split}
     \norm{ \ndist(v) \otimes \ndist(-v) - \ndist \otimes \ndist}_{TV} &\leq 1 - \frac{1}{2} \expo{-\ltwo{v}^2},
\end{split}
\end{equation}
and so assumption \ref{cond:noise_low_dim} holds with $\beta = 2+\log(2)$. From \Cref{thm:lb_low_d} it follows that:
\begin{equation}
    \inf_{\hpi} \Ex{\gm} = \Omega\left( (n^{(2 - o(1))} \sigma^d)) \wedge n^{(1 - o(1))} \right).
\end{equation}
Since $\linf{f_\idist} \leq (2\pi)^{-\frac{d}{2}}$ and $\ndist$ has bounded sub-Gaussian norm and zero-mean, assumptions \ref{cond:up_density} and \ref{cond:sub_gauss} are also satisfied. It follows from \Cref{thm:lss_up_low_d} that,

\begin{equation}
    \inf_{\hpi} \Ex{\gm} \leq \Ex{\lssm} = O\left( (n^{(2 + o(1))} \sigma^d)) \wedge n^{(1 + o(1))} \right),
\end{equation}
from which the result follows.
\end{proof}
Note that our upper bound for the Gaussian model is slightly worse than the one in \cite{geo_match_2022} because of an extra $n^{o(1)}$ factor. This is natural since our results make use of less properties of the distribution of the noise. On the other hand, if $d$ is fixed we can ignore the $n^{o(1)}$ terms in the exponent and thus obtain that the optimal error rate is of order $\Theta\left( (n^2 \sigma^d) \wedge n \right)$. This in particular confirms that the upper bound for the expected number of mistakes made by the greedy distance algorithm analyzed in Appendix A.1 of \cite{geo_match_2022} is sharp when $d$ is fixed. 

\subsection{Fixed dimension, continuous initial positions and noise}
\label{subsec:fix_dim_app}

If we take $d, \idist$ and $\ndist$ to be fixed as $n$ grows we can establish the optimal rate for a much wider class of distributions. These assumptions are very natural if one considers for instance the problem of particle tracking.
\begin{proposition}[Fixed dimension, continuous initial positions, continuous and centered noise]
\label{prop:minimax_fix_d}
    In \eqref{eq:final_pos}, let $d$ be fixed, $X_1 \sim \idist$ and $\noise_1 = \sigma \nnoise_1$ for $\sigma > 0$ and $\nnoise_1 \sim \ndist$ with $\idist$ and $\ndist$ two fixed distributions in $\R^d$ having bounded continuous densities, the latter also being zero-mean and satisfying $\Ex{\ltwo{\tilde{Z}_1}^d} < +\infty$. Then, as $n\to +\infty$,
    \begin{equation}
        \inf_{\hpi} \Ex{\gm} = \Theta\left( (n^2 \sigma^d) \wedge n \right).
    \end{equation}
\end{proposition}

\begin{proof}
Since $f_\idist$, the density of $\idist$, is continuous, it follows that for any choice of norm $\enorm$ there must exists $v \in \R^d$ and $R_d > 0$ such that:
\begin{equation}
    \inf_{\norm{x-v} < R_d} f_\idist(x) > 0.
\end{equation}
Suppose without loss of generality that $v = 0$, the case $v \neq 0$ is analogous. Then assumption \ref{cond:i_pos_low_dim} from \Cref{mod:low_dim} is satisfied with this same choice of $\enorm$ and $R_d$ and with,
\begin{equation}
    \gamma = \idist\left(B_{\enorm}(0,R_d) \right) > 0 \text{ and } \beta = \log(2) \vee \left(\frac{1}{d}\left(\log(\linf{f_\idist}) - \log\left(\inf_{\norm{x} < R_d} f_\idist(x)\right)\right)\right).
\end{equation}
Also, we have:
\begin{equation}
\label{eq:int_rep_tv}
     \norm{ \ndist(v) \otimes \ndist(-v) - \ndist \otimes \ndist}_{TV} = \frac{1}{2} \int_{\R^{d} \times \R^d } |f_\ndist(z_1 - v)f_\ndist(z_2 + v) - f_\ndist(z_1)f_\ndist(z_2)|dz_1 dz_2.
\end{equation}
Since for each $z_1, z_2 \in \R^d$, $f_\ndist(z_1 - v)f_\ndist(z_2 + v) \to f_\ndist(z_1)f_\ndist(z_2)$ as $\norm{v} \to 0$ due to the continuity of $f_\ndist$, it follows from Scheffé's Lemma that the integral on the right hand side of \eqref{eq:int_rep_tv} goes to $0$. Taking $R_d$ smaller and $\beta$ larger if necessary, assumption \ref{cond:noise_low_dim} from \Cref{mod:low_dim} also holds. Thus \Cref{thm:lb_low_d} follows and $\inf_{\hpi} \Ex{\gm} = \Omega((n^2 \sigma^d) \wedge n)$. 

\bluenew{The greedy distance algorithm discussed in \Cref{subsec:other_est} achieves an expected error of order $O(n^2\sigma^d \wedge n)$ as shown in \Cref{a:other_estim}. Hence,}
\begin{equation}
    \inf_{\hpi} \Ex{\gm} = O\left( (n^2 \sigma^d) \wedge n \right).
\end{equation}

\end{proof}

{\color{blue}
A corollary of this results is that we have the following phase transitions:
\begin{corollary}
    \label{cor:phase_trans}
    \begin{enumerate}
        \item If $\sigma \ll n^{-\frac{2}{d}}$, \textit{near-perfect recovery} $(\Ex{\gm} = o(1))$ is achievable.
        \item If $ n^{-\frac{2}{d}} \ll \sigma \ll n^{-\frac{1}{d}}$, \textit{near-perfect recovery} is impossible, but \textit{strong recovery} $(\Ex{\gm} = o(n))$ is achievable.
        \item If $ n^{-\frac{1}{d}} \ll \sigma$, \textit{strong recovery} is impossible.
    \end{enumerate}
\end{corollary}
}

\bluenew{Analogous phase transitions were already discussed in \cite{inference_particle_tracking} and in \cite{geo_match_2022}. The scale $n^{-\frac{2}{d}}$ can be understood as the scale of the minimum distance between any pair of initial positions, while the scale $n^{-\frac{1}{d}}$ provides the typical distance between any pair of nearest neighbors.}

{\color{blue}
We note that under the stronger assumption that $\tilde{Z}_1$ is sub-Gaussian, $\lss$ satisfies:
\begin{equation}
    \Ex{\lssm} = O\left( (n^2 \sigma^d) \wedge n \right)
\end{equation}
following \Cref{thm:lss_up_low_d}. Hence in this case it is minimax optimal. It is unclear to us what are the minimal assumptions on the low-dimensional regime for $\lss$ to be minimax optimal. But we believe a stronger form of tail assumptions than merely $\Ex{\ltwo{Z_1}^d} < \infty$ is required to obtain strong control over the probability of large augmenting cycles increasing the number of errors significantly.
}

Under these same assumptions for minimax optimality of $\lss$, we can also improve our constant in the upper bounds from \Cref{thm:lss_up_low_d}:
\begin{proposition}[Improved error upper bound for LSS (proof in \S \ref{sec:up_bound_ld})]
\label{thm:lss_up_fix_d} We have that for all $n \geq 2$,
\begin{equation}
     \Ex{\lssm} \leq ((\tau+o(n\sigma^d)) (n^2 \sigma^d)) \wedge n,
\end{equation}
for,
\begin{equation}
    \tau = \tau_{\idist,\ndist,d} = 2^{-d}\rho_d \Ex{f_\idist(X_1)} \Ex{\|\nnoise_1 - \nnoise_2 \|_2^d},
\end{equation}
where $\rho_d$ is the volume of $B_{\ltwo{\cdot}}(0,1)$ in $\R^d$.
\end{proposition}

\bluenew{We conjecture that our constant $\tau$ in \Cref{thm:lss_up_fix_d} is sharp in the regime where $n \sigma^d = o(1)$. This conjecture stems from two observations: first that the mass of augmenting $2$-cycles should dominates the error of the $\lss$ on the aforementioned regime, and second that, as we will in \Cref{lem:estim_aug_2_cycle}, for $\sigma$ small, $\Pr{C_2 \text{is augmenting}} \approx \tau \sigma^d$. The missing ingredient is a sharp lower bound on the error rate of $\lss$ on this regime.}

{\color{blue}
This is supported by the simulations shown in \Cref{fig:error_rate}. The interesting feature of this result, which is illustrated in the same figure, is that when $n \sigma^d = o(1)$ the error rate depends very weakly on the geometry of the distribution $\ndist$, so distributions with significantly different geometric structures give rise to an error rate similar to the Gaussian model. When $\sigma^d n^2$ starts to have order $\Omega(n)$, other geometric properties of $\idist$ and $\ndist$ seem to play crucial roles and the predicted value of $\tau n^2 \sigma^d$ ceases to be a good approximation to $\Ex{\lssm}$ as suggested by the same figure. 
We remark that although some of the distributions in \Cref{fig:error_rate} do not directly satisfy the assumptions of \Cref{thm:lss_up_fix_d}, we will see in \Cref{sec:up_bound_ld} that the same result holds under weaker assumptions which includes all examples shown. 
}
\begin{figure}[H]
\begin{subfigure}{.5\textwidth}
  \centering
  \includegraphics[scale=0.28]{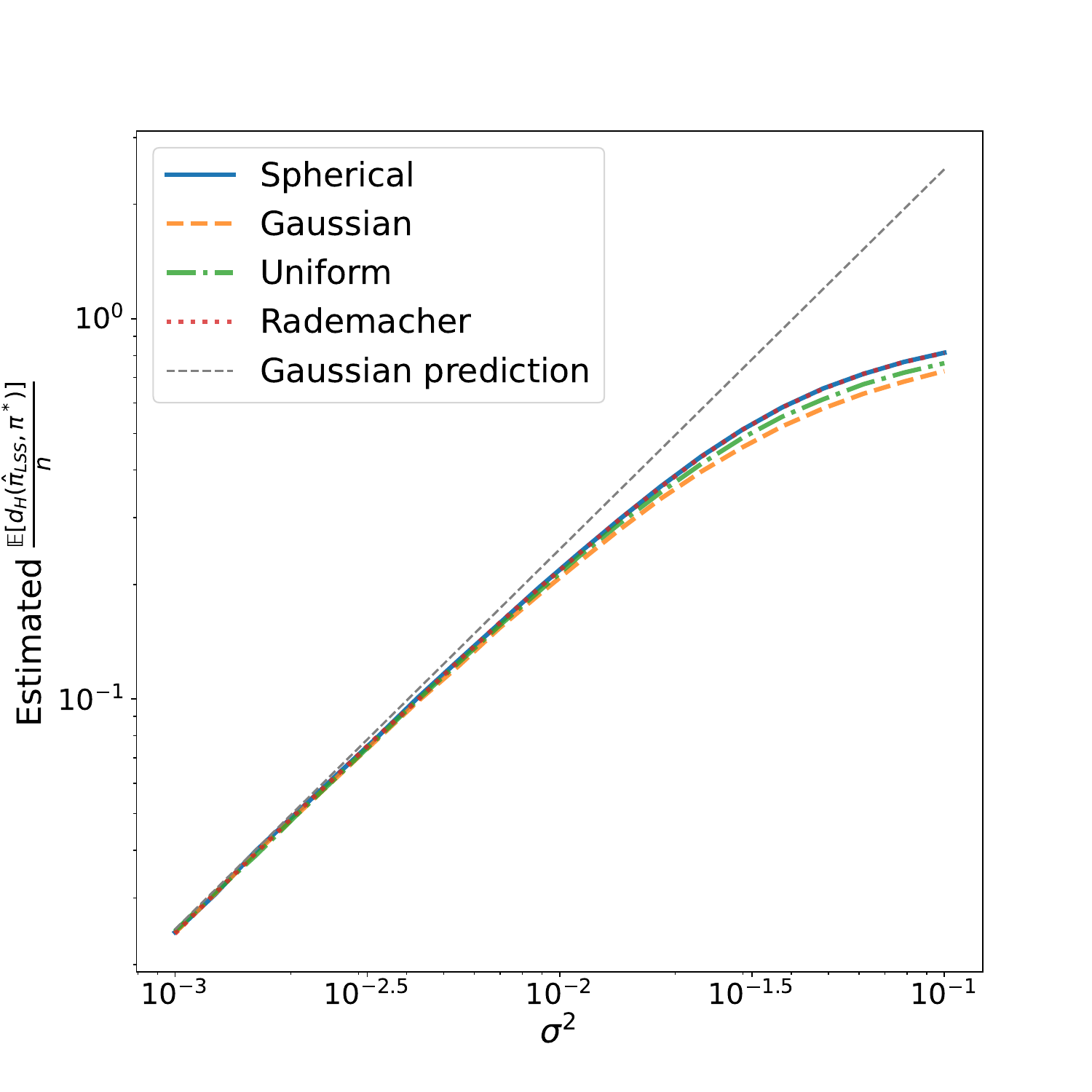}
  \caption{$d=2$}
\end{subfigure}
\begin{subfigure}{.5\textwidth}
  \centering
  \includegraphics[scale=0.28]{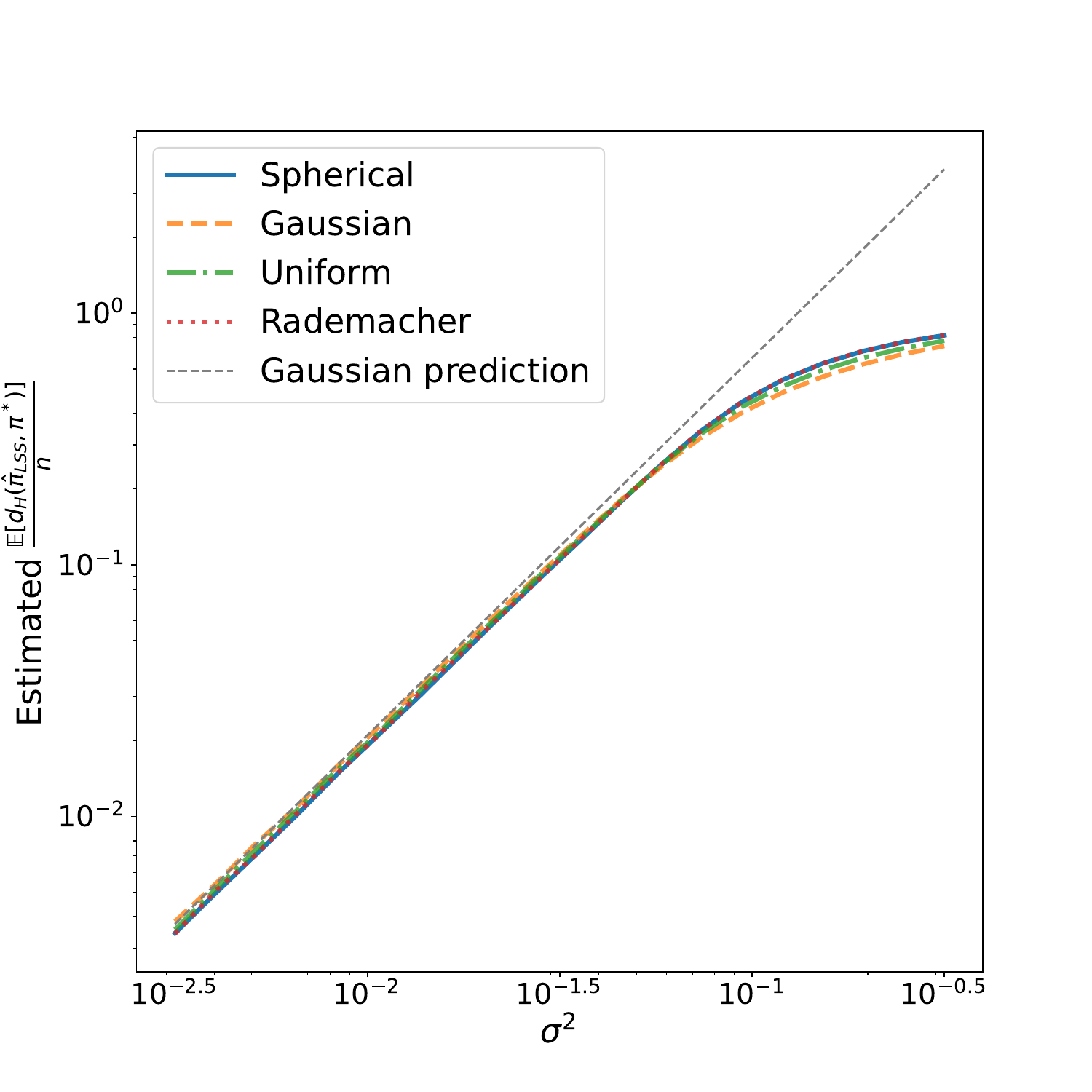}
  \caption{$d=3$}
\end{subfigure}
\caption{Simulated experiments to estimate the error rate $\frac{\Ex{\lssm}}{n}$ as a function of $\sigma^2$ for different choices of $\ndist$. We consider \eqref{eq:final_pos} with $X_1 \sim \mathcal{N}(0,I_d)$ and $\noise_1 = \sigma \nnoise_1$ with four different choices for the distribution of $\nnoise_1$: $\nnoise_1 \sim \mathcal{N}(0,I_d)$ (Gaussian), $\nnoise_1 \sim \text{Unif}(\sqrt{d}\mathbb{S}^{d-1})$ (Spherical), $\nnoise_1 \sim \text{Unif}([-\sqrt{3},\sqrt{3}]^d)$ (Uniform) and $\nnoise_1 \sim \text{Unif}(\{-1,1\}^d)$ (Rademacher). All distributions are scaled so that they lead to the same value of $\tau$ in \Cref{thm:lss_up_fix_d}. We also plot $\tau n^2 \sigma^d$ (Gaussian prediction) for this same value of $\tau$ from \Cref{thm:lss_up_fix_d}. For each choice of $\ndist$ in the simulations we considered $n=100, \, d=2,3$ and took the average of $\frac{\lssm}{n}$ over $10000$ independent trials. Both axis are plotted using log-scale. The code to reproduce the figures can also be found at \href{https://github.com/Lucas-Schwengber/particle_tracking}{https://github.com/Lucas-Schwengber/particle\_tracking}. }  
\label{fig:error_rate}
\end{figure}

\subsection{Uniformly distributed initial positions}

We now consider a model where the initial positions are uniformly distributed on the unit cube in $\R^d$ and we allow $d$ to grow slowly with $n$.

\begin{proposition}[Low-dimension, uniform initial positions, noise with i.i.d. sub-Gaussian coordinates] In \eqref{eq:final_pos}, let $d = o(\log n)$, $X_1 \sim \mathrm{Unif}([-1,1]^d)$ and $\noise_1 =  \sigma \nnoise_1$, where $\sigma > 0$,
\begin{equation}
    \nnoise_1 \sim \ndist = \bigotimes\limits_{j=1}^d \ndist',
\end{equation}
for $\ndist'$ a zero-mean, sub-Gaussian distribution in $\R$ such that:
\begin{equation}
    \sup_{|v| < 1} \tvnorm{\ndist'(v) - \ndist'} < 1-\delta,
\end{equation}
for some $\delta \in (0,1)$.
We have that as $n \to +\infty$,
\begin{equation}
    \inf_{\hpi} \Ex{\gm} = (n^{(2+ o(1))} \sigma^d) \wedge n^{(1 + o(1))} .
\end{equation}
\end{proposition}
\begin{proof}
Taking $\idist = \text{Unif}([-1,1]^d)$, assumption \ref{cond:i_pos_low_dim} from \Cref{mod:low_dim} is satisfied with $\enorm = \linf{\,\cdot\,}$, $R_d=1$, $\gamma = 1$ and $\beta = \log(2)$. Using a simple coupling argument one can show that the total variation distance of product measures satisfies:
\begin{equation}
    \left(1 - \tvnorm{\mathcal{Q}_1 \otimes \mathcal{Q}_2 - \mathcal{Q}'_1 \otimes \mathcal{Q}'_2}\right) \geq \left( 1 - \tvnorm{\mathcal{Q}_1 - \mathcal{Q}'_1} \right)\left( 1 - \tvnorm{\mathcal{Q}_2 - \mathcal{Q}'_2} \right).
\end{equation}
Since $\ndist$ from our assumption is a product measure, it follows that condition \ref{cond:noise_low_dim} from \Cref{mod:low_dim} holds with,
\begin{equation}
    \beta = \log(2) \vee \left(-2\log(\delta)\right),
\end{equation}
and so \Cref{thm:lb_low_d} holds and we have a lower bound of order $\Omega((n^{(2-o(1))} \sigma^d) \wedge n^{(1-o(1))})$. Since $\mathcal{Q}'$ is also zero-mean, sub-Gaussian and $\linf{f_\idist} \leq 2^{-d}$, the assumptions from \modlds also hold and thus the upper bound of order $O((n^{(2+o(1))} \sigma^d) \wedge n^{(1+o(1))})$ follows.
\end{proof}

 
\subsection{Noise with a log-Lipschitz density}
\label{sub:log_lip_dens}
We consider a lower bound in the regime where $d = o(\log n)$ with initial positions following a Laplace distribution and a distribution on the noise which is compatible with this assumption.

\begin{proposition}[Low-dimension, Laplace initial positions, noise with log-Lipschitz density]

In \eqref{eq:final_pos}, let $d = o(\log n)$ and assume that $X_1$ has a standard multivariate Laplace distribution, \textit{i.e.},
\begin{equation}
    X_1[1], \dots, X_1[d] \iid \mathrm{Lap}(0,1),
\end{equation}
where $\mathrm{Lap}(0,1)$ is the standard Laplace distribution on $\R$. Assume also that $\ndist$ has a continuous density $f_\ndist$ such that $\log(f_\ndist)$ is $L$-Lipschitz on the $\norm{\,\cdot\,}_1$ norm. Then as $n \to +\infty$,
\begin{equation}
    \inf_{\hpi} \Ex{\gm} = \Omega((n^{(2-o(1))} \sigma^d) \wedge n^{(1-o(1))}).
\end{equation}
\end{proposition}
\begin{proof}

Note that $|X_1[j]| \sim \text{Exp}(1)$ and so $\norm{X_1}_1$ is sub-Exponential with $\Ex{\norm{X_1}_1} = d$. It follows that condition \ref{cond:i_pos_low_dim} from \Cref{mod:low_dim} hold for sufficiently large $d$ taking $R_d = 2d, \gamma = \frac{1}{2}$ and $\beta = 2$. Since $\ndist$ has a continuous density $f_\ndist$ such that $\log(f_\ndist)$ is $L$-Lipschitz on the $\norm{\,\cdot\,}_1$ norm, we have that:
\begin{equation}
\begin{split}
    D_{KL}(\ndist \otimes \ndist || \ndist(v) \otimes \ndist(-v)) &= D_{KL}(\ndist || \ndist(v)) + D_{KL}(\ndist || \ndist(-v)) \\
    &= \int_{\R^d} (\log(f_\ndist(z)) - \log(f_\ndist(z+v))) f_\ndist(z) dz \\
    &+ \int_{\R^d} (\log(f_\ndist(z)) - \log(f_\ndist(z-v))) f_\ndist(z) dz \\
    &\leq 2L\lone{v}. 
\end{split}
\end{equation}
It follows by the Bretagnolle–Huber inequality \cite{brethuber} that,
\begin{equation}
    \norm{ \ndist(v) \otimes \ndist(-v) - \ndist \otimes \ndist}_{TV} \leq 1 - \frac{1}{2} \expo{-2L\lone{v}},
\end{equation}
and so condition \ref{cond:noise_low_dim} holds taking $\beta = 4L + \log(2)$. So \Cref{thm:lb_low_d} holds and we have a lower bound of order $\Omega\left((n^{(2-o(1))} \sigma^d) \wedge n^{(1-o(1))}\right)$.
\end{proof}

A natural choice for $\ndist$ in this case would be also a standard Laplace distribution. One might wonder if in this case the order of the expected number of mistakes obtained for $\lss$ is also nearly sharp, as $\ndist$ fails to be sub-Gaussian but it still is sub-Exponential. \bluenew{As we discuss on \Cref{subsec:other_est} and \Cref{a:other_estim}, we believe both $\lss$ and $\gdpi$ to underperform on this case, and methods such as the MLE or based on other distances to achieve better rates, matching the stated lower bound.}

\section{Lower bounds via matchings in Random Geometric Graphs}
\label{sec:lower_bounds_rggs}

In this section we prove \Cref{thm:lb_low_d} and \Cref{lem:matching_exp}. We first introduce some notation for this section. Given a finite collection of points $\mathcal{X}_n = \{x_1, \dots, x_n\}  \subset \R^d, r > 0$ and a norm $\enorm$ on $\R^d$, we define a geometric graph as a graph $G(n,r,d, \enorm) = G(\mathcal{X}_n,r,d, \enorm)$ with,
\begin{equation}
    V(G(n,r,d, \enorm)) = [n] \text{ and } E(G(n,r,d, \enorm)) = \{ \{i,j\} \subseteq [n]; ~ i \neq j \text{ and } \norm{x_i - x_j} < r \}.
\end{equation}
Following the convention of the literature, we leave the dependence of $r$ and $d$ on $n$ implicit.
If we take:
\begin{equation}
    \mathcal{X}_n = \{ X_1, \dots, X_n \} \text{ with } X_1, \dots, X_n \iid \idist,
\end{equation}
 where $\idist$ is a distribution in $\R^d$, we call the corresponding random graph $G(n,r,d,\enorm)$ a random geometric graph (RGG). As noted previously, asymptotic properties of random geometric graphs have been previously explored in several works, see for instance the book \cite{rggpenrose} for results in fixed dimension and \cite{devroye2011high} for the high-dimensional regime (\ie $d=\omega(\log n)$).

Recall that a matching in a graph $G$ is a subset of edges $M \subset E(G)$ with no vertices in common, \ie such that $e \cap e' = \emptyset$ for any pair of distinct $e, e'$ from $M$. We let $M_r = M_r(\mathcal{X}_n)$ denote the largest matching on $G(\mathcal{X}_n,r,d,\enorm)$.

\subsection{A general lower bound}

The next result gives us a general method to move from large matchings in $G(n,r,d,\enorm)$ to lower bounds on $\inf_{\hpi}\Ex{\gm}$.
For a distribution $\ndist$ over $\R^d$, and a vector $v \in \R^d$, let $\ndist(v)$ denote the law of $\nnoise + v$ where $\nnoise \sim \ndist$ and set:
\begin{equation}
\label{eq:t_def}
    t_s(\ndist, \enorm) = t_s(\ndist) := \left(1 - \sup_{\|v\| < s}  \norm{ \ndist(v) \otimes \ndist(-v) - \ndist \otimes \ndist}_{TV}\right).
\end{equation}
Our lower bound is as follows:
\begin{lemma}
\label{lem:gen_inf_theo_lb} Let $x_1,\dots,x_n\in \R^d$ be distinct. Given $\sigma>0$ set,
\begin{equation}
\label{eq:det_pos_model}
    Y_i=x_{\tpi(i)}+\sigma \nnoise_{\tpi(i)}, ~ i\in [n],
\end{equation}
where $\tpi\sim {\rm Unif}(S_n)$ is generated uniformly at random and (independently) $\nnoise_1,\dots,\nnoise_n\iid \ndist $, for $\ndist$ a distribution over $\R^d$. Then, for any estimator,
\begin{equation}
    \hpi = \hpi(x_1,\dots x_n,Y_1,\dots,Y_n)\in S_n,
\end{equation}
the following holds:
\begin{equation}
\label{eq:gen_inf_lb}
    \Ex{\hamdist{\hpi}{\tpi}} \geq \frac{|M_r|t_{\frac{r}{\sigma}}(\ndist)}{2},
\end{equation}
for $t_{s}(\ndist)$ as in \eqref{eq:t_def} above and $M_r = M_r(\{x_1,\dots,x_n\})$ the largest matching in $G(\{x_1, \dots, x_n \}, r, d, \enorm)$.
\end{lemma}

The Lemma above gives a lower bound which takes into account two effects: $|M_r|$ accounts for pairs of close initial positions, while $t_{\frac{r}{\sigma}}(\ndist)$ accounts for the difficulty of distinguishing their associated noisy versions. The free parameter $r$ controls the trade-off between these two quantities: as $r$ increases we have more matched initial positions but their noisy versions become more distinguishable. 

\begin{proof}[Proof of \Cref{lem:gen_inf_theo_lb}]
Given $v \in \R^d$ and $\sigma > 0$ we let $\ndist(v, \sigma)$ denote the law of $v + \sigma\nnoise$ for $\nnoise \sim \ndist$. In particular, $\ndist(v) = \ndist(v,1)$.
We assume for ease of notation that $M_r = \{ \{ 2j-1,2j \} \}_{j=1}^m$, but the general case follows analogously. Letting $\Omega = \{0, 1 \}^m$, we associate to each element of $\Omega$ a permutation $\pi_{\omega} \in S_n$ given by:

\begin{equation}
    \bluenew{\pi_{\omega} = (1 ~ 2)^{\omega_1} ~ (3 ~ 4)^{\omega_2} ~ \cdots ~ ((2m-1) ~ 2m)^{\omega_m}.}
\end{equation}

In other words, $\pi_\omega$ swaps $(2j-1)$ and $2j$ if $\omega_j = 1$ and does not swap them if $\omega_j = 0$. All other indices are left unchanged. Denote by $S_\Omega \subseteq S_n$ the subset of permutations of this form. We associate to each $\omega \in \Omega$ a probability measure $P_{\omega}$ on $(\R^d)^{n}$ given by:
\begin{equation}
    P_{\omega} = \left(\bigotimes_{j=1}^m \left( \ndist\left(x_{\pi_\omega(2j-1)},\sigma\right) \otimes \ndist\left(x_{\pi_\omega(2j)},\sigma\right)\right)\right) \otimes \bigotimes_{k=2m+1}^n \ndist\left(x_k,\sigma\right). 
\end{equation}
That is, letting $Y_1, \dots, Y_n$ be random vectors with joint distribution $P_{\omega}$, we have that for $j \in [m]$, if $\omega_j = 0$ then,
\begin{equation}
    Y_{2j-1} = x_{2j-1} + \sigma \nnoise_{2j-1}, \  Y_{2j} = x_{2j} + \sigma \nnoise_{2j},
\end{equation}
if $\omega_j = 1$ then,
\begin{equation}
    Y_{2j-1} = x_{2j} + \sigma \nnoise_{2j-1}, \  Y_{2j} = x_{2j-1} + \sigma \nnoise_{2j},
\end{equation}
and for $k = 2m+1, \dots, n$,
\begin{equation}
    Y_k = x_k + \sigma \nnoise_k,
\end{equation}
where $\nnoise_1, \dots, \nnoise_n \iid \ndist$. Note that if $\omega$ and $\omega'$ differ by a single coordinate $j$ then:
\begin{equation}
\begin{split}
\label{eq:tv_bound}
    \norm{ P_{\omega} - P_{\omega'} }_{TV} &= \norm{ \ndist\left(x_{\pi_\omega(2j-1)},\sigma\right) \otimes \ndist\left(x_{\pi_\omega(2j)}, \sigma \right) - \ndist\left(x_{\pi_{\omega'}(2j-1)},\sigma\right) \otimes \ndist\left(x_{\pi_{\omega'}(2j)}, \sigma \right)}_{TV} \\
    &= \norm{ \ndist\left( \frac{x_{\pi_\omega(2j-1)}}{\sigma} \right) \otimes \ndist\left( \frac{x_{\pi_\omega(2j)}}{\sigma} \right) - \ndist\left( \frac{x_{\pi_{\omega'}(2j-1)}}{\sigma} \right) \otimes \ndist\left( \frac{x_{\pi_{\omega'}(2j)}}{\sigma} \right)}_{TV} \\
    &\leq \sup_{ \norm{v} < \frac{r}{\sigma}} \norm{\ndist(v) \otimes \ndist(-v) - \ndist\otimes \ndist}_{TV} = 1- t_{\frac{r}{\sigma}}(\ndist)
\end{split}
\end{equation}
where the second equality and third inequality follow since the total variation distance is invariant under translations and dilations respectively and in the last inequality we also used the fact that $\| x_{2j} - x_{2j-1} \| < r$.
Suppose we observe $x_1, \dots, x_n \in \R^d$ and $Y_1, \dots, Y_n \sim P_{\omega}$ from an unknown $\omega$. Given $\hat{\omega}$ an estimator for $\omega$ based on $x_1, \dots, x_n$ and $Y_1, \dots, Y_n$, it follows by Assouad's Lemma \cite{tsybakov2008introduction} and \eqref{eq:tv_bound} that:
\begin{equation}
    \inf_{\hat{\omega}} \max_{\omega \in \Omega} \Exs{\omega}{ \hamdist{\hat{\omega}}{\omega} } \geq \frac{m t_{\frac{r}{\sigma}}(\ndist)}{2},
\end{equation}
where $\Exs{\omega}{\, \cdot \, }$ denotes expectations under $P_{\omega}$ and the infimum is taken over all (possibly non-deterministic) estimators $\hat{\omega}$. 

We can convert any estimator $\hat{\pi}$ for a hidden permutation $\tpi$ based on $x_1, \dots, x_n$ and $Y_1, \dots, Y_n$ into an estimator $\hat{\omega}^{(\hpi)}$ for $\omega$ by setting:
\begin{equation}
\hat{\omega}^{(\hpi)}_j =
    \begin{cases}
         1, ~ \text{if $\hat{\pi}(2j-1) = 2j$ and $\hat{\pi}(2j) = 2j-1$} \\
         0, ~ \text{if $\hat{\pi}(2j-1) = 2j-1$ and $\hat{\pi}(2j) = 2j$} \\
         \epsilon, ~ \text{otherwise}
    \end{cases},
\end{equation}
where $\epsilon$ indicates an error regardless of the true $\omega_j$. For any $\omega \in \Omega$, $\hat{\omega}^{(\hpi)}_j \neq \omega_j$ implies that $\hat{\pi}(2j) \neq \pi_{\omega}(2j)$ or $\hat{\pi}(2j-1) \neq \pi_{\omega}(2j-1)$. It follows that:
\begin{equation}
\begin{split}
    \hamdist{\hat{\pi}}{\pi_{\omega}} &= \sum_{i=1}^n \indic{\hat{\pi}(i) \neq \pi_{\omega}(i)} \\
    &\geq \sum_{j=1}^m \indic{\hat{\pi}(2j) \neq \pi_{\omega}(2j)} + \indic{\hat{\pi}_{\omega}(2j-1) \neq \pi_{\omega}(2j-1)} \\
    &\geq \hamdist{\hat{\omega}^{(\hpi)}}{\omega},
\end{split}
\end{equation}
and so,
\begin{equation}
    \inf_{\hat{\pi}} \max_{\pi_\omega \in S_\Omega} \Ex{ \hamdist{\hpi}{\pi_\omega} | \tpi = \pi_\omega} \geq \frac{m t_{\frac{r}{\sigma}}(\ndist)}{2},
\end{equation}
where expectations are taken with respect to the model in \eqref{eq:det_pos_model}.
    Since $S_\Omega \subseteq S_n$,
\begin{equation}
\label{eq:assouad_sn}
    \inf_{\hpi} \max_{\pi \in S_n} \Ex{ \hamdist{\hpi}{\pi} | \tpi = \pi} \geq \frac{m t_{\frac{r}{\sigma}}(\ndist)}{2}.
\end{equation}
This proves the minimax lower bound. Note that we could take $\hpi$ to be a non-deterministic estimator and the argument still follows. Also, due to the structure of the problem, we always have that:
\begin{equation}
\label{eq:minimax_bayes_id}
    \inf_{\hpi} \max_{\pi \in S_n} \Ex{ \hamdist{\hpi}{\pi} | \tpi = \pi} = \inf_{\hpi} \Ex{ \hamdist{\hpi}{\pi^*}},
\end{equation}
where the infimum is taken over all (possibly non-deterministic) estimators $\hpi$. We provide a proof of \eqref{eq:minimax_bayes_id} on \Cref{a:estim_group} in a more general setting. From \eqref{eq:minimax_bayes_id} the result follows. 
\end{proof}

\subsection{Matchings in random geometric graphs}
\label{sec:largest_matching_rgg}

Our next step to use \Cref{lem:gen_inf_theo_lb} to prove \Cref{thm:lb_low_d} is proving \Cref{lem:matching_exp}. As in the Introduction, we consider a random geometric graph $G(n,r,d, \enorm) = G(\mathcal{X}_n,r,d, \enorm)$ induced by $\mathcal{X}_n = \{X_1, \dots, X_n \}$, where $X_1, \dots, X_n \iid \idist$ a distribution with density $f_\idist$ satisfying assumption \ref{cond:i_pos_low_dim} from \Cref{mod:low_dim}. The idea of the proof is to build in a matching in $G(n,r,d, \enorm)$ by considering a collection of non-overlapping balls of radius $\frac{r}{2}$ inside $B_{\enorm}\left(0,R_d\right)$ and considering the edges between pairs of points that fall inside the same ball.

\begin{proof}[Proof of  \Cref{lem:matching_exp}]
Given $n \geq 3$ and $r > 0$, let
\begin{equation}
    l = \frac{2 R_d e^{-\beta}}{(n-1)^{\frac{1}{d}}}.
\end{equation}
Assume first that $r \leq l$, which implies in particular that $r < R_d$ since $\beta \geq \log(2)$. Let $B_d = B_{\enorm}\left(0,R_d\right) \subseteq \R^d$, which is such that $\idist\left( B_d \right) \geq \gamma > 0$ by assumption. Let $k_n$ be the maximum number of non-overlapping (open) balls of radius $\frac{r}{2}$ (on norm $\enorm$) entirely contained in $B_d$. Since $r<R_d$, $k_n$ is at least the size of the largest $r$-packing in $B_{\enorm}\left(0,\frac{R_d}{2}\right)$. Following the volumetric lower bound for the size of the $r$-packing number in \cite[Theorem 14.2]{wu_inf_stats}, we obtain
\begin{equation}
 \label{eq:bound_kn}
     \left(\frac{R_d}{2r}\right)^d \leq k_n.
 \end{equation}
Let $\mathcal{K}_n = \{B_{i,d}\}_{i=1}^{k_n}$ be a collection of $k_n$ non-overlapping balls of radius $\frac{r}{2}$ fully contained in $B_d$. Denote by $N = |\mathcal{X}_n \cap B_d|$ the number of samples from $\mathcal{X}_n = \{X_1, \dots, X_n\}$ that fall inside $B_d$, and define
\begin{equation}
    U_{i,d} = \sum_{j=1}^n \indic{X_j \in B_{i,d}} = \left|\mathcal{X}_n \cap B_{i,d}\right|, ~ i \in [k_n],
\end{equation}
the number of points from $\mathcal{X}_n$ that fall inside $B_{i,d}$. It is clear that, conditionally on $N$, $U_{i,d} \sim \Bin(N, p_{i,d})$, where $p_{i,d} = \Pr{X_1 \in B_{i,d}|X_1 \in B_d}$ for each $i \in [k_n]$. Moreover,
\begin{equation}
\label{eq:ball_prob_bound}
    e^{-\beta d} \left(\frac{r}{2R_d}\right)^d \leq p_{i,d} \leq e^{\beta d}\left(\frac{r}{2R_d}\right)^d,
\end{equation}
since $\frac{\text{ess sup}_{x \in B_d}f_\idist(x)}{\text{ess inf}_{x \in B_d}f_\idist(x)} \leq e^{\beta d}$ by assumption. Let
\begin{equation}
    I_n = \sum_{i=1}^{k_n} \indic{U_{i,d}=2},
\end{equation}
be the number of balls $B_{i,d}\in \mathcal{K}_n$ that have exactly $2$ elements from $\mathcal{X}_n$. Note that $|M_r| \geq I_n$ since we can build a matching in $G(n,r,d,\enorm)$ by considering the edges between each pair of points on the same $B_{i,d}$ such that $U_{i,d}=2$. Moreover, for $k\in \{0\}\cup[n]$,
\begin{equation}
\begin{split}
    \Pr{U_{i,d}=2|N=k} & =\binom{k}{2} p^2_{i,d}(1-p_{i,d})^{k-2} \\ &\geq \binom{k}{2} p^2_{i,d}(1-p_{i,d})^{n-2}  \\ 
    &\geq \binom{k}{2} \frac{e^{-2\beta d}}{4} \left(\frac{r}{2R_d}\right)^{2d},
\end{split}
\end{equation}
where we used \eqref{eq:ball_prob_bound} to bound $p_{i,d}$ above and below and also $e^{\beta d} \left( \frac{r}{2R_d}\right)^d \leq \frac{1}{n-1}$ by our assumption that $r \leq l$, and the fact that $(1-\frac{1}{(n-1)})^{n-1} \geq \frac{1}{4}$ if $n \geq 3$. It follows that,
\begin{equation}
\begin{split}
    \Pr{U_{i,d}=2} &= \Ex{\Pr{U_{i,d}=2|N}} \\
    &\geq \Ex{\binom{N}{2}} \frac{e^{-2\beta d}}{4} \left(\frac{r}{2R_d}\right)^{2d} \\
    &\geq \gamma^2 \binom{n}{2} \frac{e^{-2\beta d}}{4} \left(\frac{r}{2R_d}\right)^{2d},
\end{split}
\end{equation}
where in the last inequality we used the fact that $N \sim \Bin(n,\Pr{X_1 \in B_d})$ implies $\Ex{\binom{N}{2}} = (\Pr{X_1 \in B_d})^2 \binom{n}{2}$ and also $\Pr{X_1 \in B_d} \geq \gamma$ by assumption.
So,
\begin{equation}
\begin{split}
    \Ex{|M_r|} &\geq \Ex{I_n} \\
    &\geq k_n \gamma^2  \binom{n}{2} \frac{e^{-2\beta d}}{4} \left(\frac{r}{2R_d}\right)^{2d} \\
    &\geq \gamma^2  \binom{n}{2} \frac{e^{-2\beta d}}{4} \left(\frac{r}{8R_d}\right)^{d},
\end{split}
\end{equation}
where the last inequality follows from \eqref{eq:bound_kn}. If $r > l$, then $|M_r| \geq |M_l|$, and so,
\begin{equation}
\begin{split}
\Ex{|M_r|} &\geq \Ex{|M_l|} \\
&\geq \gamma^2 \binom{n}{2} \frac{e^{-2\beta d}}{4} \left(\frac{l}{8R_d}\right)^{d} \\
&= \gamma^2 \frac{n}{2} \frac{e^{-2\beta d}}{4} e^{-\beta d} 4^{-d} \\
&= \gamma^2 4^{-d} n \frac{e^{-3\beta d}}{8}.
\end{split}
\end{equation}
Since $\binom{n}{2} \geq \frac{n^2}{4}$ for all $n \geq 2$, it follows that for a general $r > 0$,
\begin{equation}
    \Ex{|M_r|} \geq \frac{\gamma^2}{16} e^{-d(3\beta+\log(8))} \left(n^2 \left( \frac{r}{R_d} \right)^d \wedge n \right), 
\end{equation}
and so the result follows.
\end{proof}

The size of the largest matching $M_r$ in $G(n,r,d, \enorm)$ cannot exceed $\lfloor n/2\rfloor$ or the total number of edges of $G(n,r,d,\enorm)$, $|E(G(n,r,d,\enorm))|$. This implies that:
\begin{equation}
    \Ex{|M_r|}\leq \Ex{|E(G(n,r,d, \enorm))|} \wedge \left\lfloor \frac{n}{2}\right\rfloor.
\end{equation}
A simple calculation shows that this implies:
\begin{equation}
    \Ex{|M_r|}\leq \left(\linf{f_\idist} \text{Vol}(B_{\enorm}(0,1)) \,r^d\binom{n}{2}\right)\wedge\left\lfloor \frac{n}{2}\right\rfloor.
\end{equation}
It follows that for fixed $d$, the order from \Cref{lem:matching_exp} is sharp. 

\subsection{Minimax lower bound}\label{sub:minimaxlowerbound}

We now prove \Cref{thm:lb_low_d}.
\begin{proof}[Proof of \Cref{thm:lb_low_d}]
    Under \Cref{mod:low_dim}, using our assumptions on $\idist$ and $\ndist$ and applying \Cref{lem:matching_exp} to \Cref{lem:gen_inf_theo_lb} with $r = R_d \sigma$, we have that whenever $n \geq 3$,
    \begin{equation}
        \inf_{\hpi}\Ex{\gm} \geq \frac{\gamma^2}{32} e^{-7\beta d} \left( (n^2 \sigma^d) \wedge n \right).
    \end{equation}
\end{proof}

\section{Upper bounds}
\label{sec:up_bounds}

In this section we provide the technical results necessary to prove the upper bounds in \Cref{thm:lss_up_low_d}, \Cref{thm:lss_up_fix_d}, \Cref{thm:high_dim_lss} and \Cref{thm:high_dim_lssc}.

We first note that conditioned on $\tpi = \pi$, for some fixed $\pi \in S_n$, the law of $\lssm$ is the same regardless of the choice of $\pi$ \cite{geo_match_2022}. So in this section, unless stated otherwise, we assume without loss of generality that $\tpi = \text{Id}$. Under this condition we will consider a $t$-cycle $C_t = (i_1, \dots, i_k)$ to be augmenting for $\lss$ if and only if:
\begin{equation}
    \sum_{k=1}^t \langle X_{i_k}, Y_{i_{k+1}} \rangle \geq \sum_{k=1}^t \langle X_{i_k}, Y_{i_k} \rangle.
\end{equation}

The main content of the Lemmata we use to prove our upper bounds are bounds on the probability that $t$-cycles are augmenting. Taking expectations on \eqref{eq:det_bound_lss} we can use these bounds to bound $\Ex{\lssm}$.

\subsection{Upper bounds for low-dimension}
\label{sec:up_bound_ld}

Recall that a random variable $U$ is sub-Gaussian \cite{vershynin_2018} if $\snorm{ U } < +\infty$.
Moreover, a random vector $V$ in $\R^d$ is sub-Gaussian if:
\begin{equation}
    \snorm{ V } := \sup_{\|v\|=1} \snorm{ \langle V, v \rangle } < + \infty.
\end{equation}
We have the following bound on the probability that a given $t$-cycle is augmenting:
\begin{lemma}
\label{lem:aug_up_bound_fix_d} Under \modlds \hspace{-0.75em} , for any $t$-cycle $C_t$,
    \begin{equation}
        \Pr{ C_t \text{ is augmenting}} \leq 2\left( K \sigma \right)^{d(t-1)},
    \end{equation}
    for $K = b_{\mathrm{LD}} K_\ndist^2 e^{2\beta}$ where $b_{\mathrm{LD}} > 0$ is an absolute constant.
\end{lemma}

\bluenew{The proof follows very similar lines as \cite{geo_match_2022}. Both rely on bounding an expectation of the form $\Ex{e^{-\ltwo{Q_t \ubar{X}}^2}}$, where $\ubar{X} \in \R^{t \times d}$ is a concatenation of the first $t$ initial positions, and $Q_t$ is a square matrix of rank $(t-1) \times d$. The Gaussian assumptions in \cite{geo_match_2022} allow the authors to carry out explicit computations by diagonalizing the quadratic form and using the rotation invariance of the standard Gaussian. We instead use a bound of the form $\Ex{e^{-\ltwo{Q_t \ubar{X}}^2}} \leq \Ex{e^{-\ltwo{\tilde{Q}_t \ubar{\tilde{X}}}^2}}$, where $\ubar{\tilde{X}} \in \R^{(t-1) \times d}$ is the image of $\ubar{X}$ under a (full rank) linear transformation and $\tilde{Q}_t$ is invertible. This allows us to carry out computations in terms of the density of $\ubar{\tilde{X}}$ and obtain a final bound in terms of a Gaussian integral in dimension $(t-1) \times d$.}

\begin{proof}
Let $\Delta_1 := X_t - X_1$, $\Delta_i := X_{i-1} - X_i, ~ i = 2, \dots, t$, $\ubar{Z} \in \R^{t \times d}$ be a matrix with entries $\ubar{Z}[i,j] = Z_i[j]$, \ie the matrix whose rows are given by $Z_i$, and define the matrices $\ubar{\nnoise}, ~ \ubar{\Delta}, ~ \ubar{X} \in \R^{t \times d}$ analogously. We first assume that $C_t = (1,2,\ldots,t)$. It is easy to see that $C_t$ will be augmenting if and only if,
\begin{equation}
\label{eq:aug_cycle_lss}
\sum_{i=1}^t \langle \Delta_i, Z_i \rangle \geq \frac{1}{2} \left(\sum_{i=1}^t \lVert \Delta_i \rVert_2^2  \right).
\end{equation}
This is equivalent to,
\begin{equation}
\langle \ubar{Z}, \ubar{\Delta} \rangle \geq \frac{1}{2} \| \ubar{\Delta} \|_2^2,
\end{equation}
and,
\begin{equation}
\left\langle \ubar{\nnoise}, \frac{\ubar{\Delta}}{\ltwo{ \ubar{\Delta} }} \right\rangle \geq \frac{1}{2} \frac{\ltwo{\ubar{\Delta}}}{\sigma}.
\end{equation}
This last step is valid since $\Pr{\ltwo{ \ubar{\Delta} } = 0} = 0$, as $\idist$ has no atoms, so we can also assume that $\ltwo{ \ubar{\Delta} } \neq 0$.
Let $v_1, \dots, v_t \in \R^d$ and $\ubar{v} \in \R^{t \times d}$ with $\ubar{v}[i,j] = v_i[j]$ and $\ltwo{\ubar{v}} = 1$. We have that for some absolute constant $\kappa > 0$ and $K_\ndist$,
\begin{equation}
\begin{split}
 \snorm{\left\langle \ubar{\nnoise}, \ubar{v} \right\rangle}^2 &=  \snorm{ \sum_{i=1}^t \sum_{j=1}^d \ubar{\nnoise}[i,j] \ubar{v}[i,j]}^2 \\
 &\leq \kappa \sum_{i=1}^t \snorm{ \sum_{j=1}^d  \nnoise_i[j] v_i[j]  }^2 \\
 &= \kappa \sum_{i=1}^t \indic{\ltwo{v_i} \neq 0} \ltwo{v_i}^2 \snorm{ \left\langle \nnoise_i, \frac{v_i}{\ltwo{v_i}} \right\rangle }^2 \\
 &\leq \kappa K_\ndist^2.
\end{split}
\end{equation}
For any sub-Gaussian random variable $U$ we have that,
\begin{equation}
    \mathbb{P}\left(|U| \geq t \right) \leq 2e^{-\frac{ t^2}{\snorm{U}^2}}.
\end{equation}
It follows that since $\ubar{\nnoise}$ and $\ubar{\dpos}$ are independent,
\begin{equation}
\begin{split}
    \Pr{\binner{ \ubar{\nnoise}}{\frac{\ubar{\dpos}}{\ltwo{\ubar{\Delta}}} } \geq \frac{1}{2} \frac{\ltwo{\ubar{\Delta}}}{\sigma}  \Big\rvert \ubar{\Delta}} &\leq 2 \exp\left(-\frac{1}{4\kappa K^2_\ndist}\frac{\ltwo{ \ubar{\Delta} }^2}{\sigma^2}\right) \\
    &= 2 \exp\left(-\frac{1}{4\kappa K^2_\ndist}\frac{ \sum\limits_{i=1}^t \ltwo{ \Delta_i }^2}{\sigma^2}\right) \\
    &\leq 2 \exp\left(-\frac{1}{4\kappa K^2_\ndist}\frac{ \sum\limits_{i=1}^{t-1} \ltwo{ \Delta_i }^2}{\sigma^2}\right),
\end{split}
\end{equation}
where in the last inequality we used the fact that $\ltwo{\Delta_t}^2 \geq 0$. Although this last step might seem arbitrary we will see that it allows us to bound the expectation of the expressions above by a Gaussian integral. Remember that $f_\idist$ denotes the density of $X_1$ with respect to Lebesgue measure in $\R^d$. Clearly $\ubar{X}$ also has a density with respect to Lebesgue measure in $\R^{t \times d}$, but $\ubar{\Delta}$ does not, since $\sum_{i=1}^{t-1} \dpos_i = -\dpos_t$ and thus $\ubar{\Delta}$ takes values in a $((t-1) \times d)$-dimensional sub-space of $\R^{t \times d}$. However, setting $\ubar{\Delta}_{(t-1)} \in \R^{(t-1) \times d}$ to be $\ubar{\Delta}$ without its $t$-th row, we have that $\ubar{\Delta}_{(t-1)}$ has a density with respect to the Lebesgue measure on $\R^{(t-1) \times d}$ given by: 

\begin{equation}
    f_{\ubar{\Delta}_{(t-1)}}\left( b \right) = \int_{\R^d} \prod_{j=1}^{(t-1)} f_\idist \left(x - \sum_{i=1}^{j} b[i,:] \right) f_\idist(x)dx, ~ b \in \R^{(t-1) \times d}.
\end{equation}
In particular,
\begin{equation}
    \linf{f_{\ubar{\Delta}_{(t-1)}}} \leq \int_{\R^d} \linf{f_\idist}^{(t-1)} f_\idist(x) dx = \linf{f_\idist}^{(t-1)}.
\end{equation}
It follows that, 
\begin{equation}
\begin{split}
    \Pr{ C_t \text{ is augmenting} } &= \Ex{ \Pr{\binner{ \ubar{\nnoise}}{\frac{\ubar{\dpos}}{\ltwo{\ubar{\Delta}}} } \geq \frac{1}{2} \frac{\fnorm{\ubar{\Delta}}}{\sigma}  \Big\rvert \ubar{\Delta}} } \\
    &\leq 2 \int_{\R^{(t-1) \times d}} \exp\left(-\frac{1}{4\kappa K^2_\ndist}\frac{ \sum\limits_{i=1}^{t-1} \ltwo{ b[i,:] }^2}{\sigma^2}\right) f_{\ubar{\Delta}_{(t-1)}}\left( b \right) db \\
    &\leq 2 \| f_{\ubar{\Delta}_{(t-1)}} \|_{\infty} \int_{\R^{(t-1) \times d}} \exp\left(-\frac{1}{4\kappa K^2_\ndist}\frac{ \sum\limits_{i=1}^{t-1} \ltwo{ b[i,:] }^2}{\sigma^2}\right) db \\
    &\leq 2 \linf{f_\idist}^{(t-1)} \left( 4 \pi \kappa K^2_\ndist \sigma^2 \right)^{\frac{d(t-1)}{2}},
\end{split}
\end{equation}
and so the result follows for $C_t = (1,2,\ldots,t)$ taking $b_{\mathrm{LD}} = 4\pi \kappa$. Since both $X_1, \dots, X_n$ and $Z_1, \dots, Z_n$ have exchangeable distributions, the same argument follows if we take $C_t$ to be any other $t$-cycle.
\end{proof}

We now proceed to prove the upper bound from \Cref{thm:lss_up_low_d}.

\begin{proof}[Proof of \Cref{thm:lss_up_low_d}]
\label{proof:up_lss_fix_d}

Using \eqref{eq:det_bound_lss} and \Cref{lem:aug_up_bound_fix_d}, we have that,
\begin{align}    
\label{eq:ex_bound_fix_d}
    \Ex{\lssm} &\leq \sum_{t=2}^n t \left|\{ C_t; ~ C_t \text{ is a } t\text{-cycle} \} \right|\Pr{(1,\dots,t) \text{ is augmenting}} \\
    &\leq \sum_{t=2}^n n^t \Pr{(1,\dots,t) \text{ is augmenting}} \label{eq:nt_bound} \\ \label{eq:nt_bound2}
    &\leq 2 K^d n^2 \sigma^d \sum_{t=0}^{\infty} (K^dn\sigma^d)^{t}
\end{align}
where the second inequality follows since $t \left|\{ C_t; ~ C_t \text{ is a } t\text{-cycle} \} \right| = t \binom{n}{t}(t-1)! \leq n^t$. 
The upper bound in \eqref{eq:nt_bound2} implies,
\[ \Ex{\lssm}\leq  3 K^dn^2 \sigma^d\leq n\mbox{ when } K^dn\sigma^d \leq \frac{1}{3}.\]
On the other hand,  when $K^dn\sigma^d>1/3$ we still have the trivial upper bound $\Ex{\lssm}\leq n$ because no estimator makes more than $n$ mistakes. Combining these estimates finishes the proof.\end{proof}

The proof of the upper bound from \Cref{thm:lss_up_fix_d} is very similar but we need the following additional result:

\begin{lemma}
\label{lem:estim_aug_2_cycle}
Let $n \geq 2$, $C_2 = (i,j), ~ i,j \in [n]$ be any $2$-cycle and assume we are under \modlds with $d$ fixed. Then,
    \begin{equation}
    \label{eq:aug_2_cycle_lim}
        \lim_{\sigma \to 0^+} \frac{\Pr{C_2 \text{ is augmenting}}}{\sigma^d} = 2^{-d}\rho_d \Ex{f_\idist(X_1)} \Ex{\ltwo{\nnoise_1-\nnoise_2}^d},
    \end{equation}
    where $\rho_d$ is the volume of $B_{\ltwo{\,\cdot\,}}(0,1)$.
\end{lemma}
\begin{proof}
Set $\noise_1 = \sigma \nnoise_1$ and $ \noise_2 = \sigma \nnoise_2$ where $\sigma > 0$, $\nnoise_1, \nnoise_2 \iid \ndist$.
In view of \eqref{eq:aug_cycle_lss},
    \begin{equation}
    \label{eq:aug_ball_equiv}
    \begin{split}
        \Pr{C_2 \text{ is augmenting}} &= \Pr{\inner{\noise_1 - \noise_2}{\ipos_1 - \ipos_2} \geq \ltwo{\ipos_1 - \ipos_2}^2} \\ 
        &= \Pr{\ltwo{\frac{\noise_1-\noise_2}{2}} \geq \ltwo{\frac{\noise_1-\noise_2}{2} - (\ipos_1 - \ipos_2)}} \\
        &= \Pr{\ltwo{\Delta_Z} \geq \ltwo{ \Delta_Z - \Delta_X}} \\
        &= \Pr{\ltwo{\Delta_Z} \geq \ltwo{ \Delta_Z - \Delta_X}, ~ \ltwo{\Delta_Z} > 0}
        \end{split}
    \end{equation}
where $\Delta_Z = \frac{\noise_1-\noise_2}{2}$ and $\Delta_X = \ipos_1 - \ipos_2$ and the last equality follows since $\Pr{\ltwo{\Delta_X} = 0} = 0$ as $\idist$ has no atoms. Let $f_{\Delta_X}$ be the density of $\Delta_X$, which exists and is continuous since $f_\idist$ is. We can rewrite \eqref{eq:aug_ball_equiv} as:
\begin{equation}
    \Ex{\rho_d\ltwo{\Delta_Z}^d \frac{\Pr{\ltwo{\Delta_Z} \geq \ltwo{ \Delta_Z - \Delta_X}  | \Delta_Z}}{\rho_d\ltwo{\Delta_Z}^d} \indic{\ltwo{\Delta_Z} > 0}}.
\end{equation}
Note that,
\begin{equation}
\begin{split}
\Pr{\ltwo{\Delta_Z} \geq \ltwo{ \Delta_Z - \Delta_X} 
     | \Delta_Z} &= \int_{B_{\ltwo{\,\cdot\,}}\left(\Delta_Z,\ltwo{\Delta_Z}\right)} f_{\Delta_X}(x)dx \text{ a.s.}
\end{split} 
\end{equation}
Now, as $\sigma \to 0$, $\ltwo{\Delta_Z} \asc{n} 0$. Since $f_{\Delta_X}$ is continuous it follows that,
\begin{equation}
    \lim_{\sigma \to 0} \frac{\Pr{\ltwo{\Delta_Z} \geq \ltwo{ \Delta_Z - \Delta_X}  | \Delta_Z}}{\rho_d\|\Delta_Z\|^d}\indic{\ltwo{\Delta_Z} > 0} = f_{\Delta_X}(0)\indic{\ltwo{\Delta_Z} > 0} \text{ a.s.}
\end{equation}
In particular, since $\frac{\ltwo{\Delta_Z}}{\sigma} = \ltwo{\frac{\nnoise_1-\nnoise_2}{2}}$,
\begin{equation}
\label{eq:cond_limit}
    \lim_{\sigma \to 0} \frac{\rho_d\ltwo{\Delta_Z}^d}{\sigma^d}\frac{\Pr{\ltwo{\Delta_Z} \geq \ltwo{ \Delta_Z - \Delta_X}  | \Delta_Z}}{\rho_d\ltwo{\Delta_Z}^d}\indic{\ltwo{\Delta_Z} > 0} = 2^{-d}\rho_d \ltwo{\nnoise_1-\nnoise_2}^d f_{\Delta_X}(0)\indic{\ltwo{\Delta_Z} > 0}  \text{ a.s.}
\end{equation}
Note that $f_{\Delta_X}(x) = \int_{\R^d} f_\idist(y+x)f_\idist(y)dy$, so,
\begin{equation}
\frac{\rho_d\ltwo{\Delta_Z}^d}{\sigma^d}\frac{\Pr{\ltwo{\Delta_Z} \geq \ltwo{ \Delta_Z - \Delta_X}  | \Delta_Z}}{\rho_d\ltwo{\Delta_Z}^d}\indic{\ltwo{\Delta_Z} > 0} \leq \rho_d \ltwo{\frac{\nnoise_1-\nnoise_2}{2}}^d\|f_\idist\|_{\infty}\indic{\ltwo{\Delta_Z} > 0}.  
\end{equation}
By our sub-Gaussian assumption on $\ndist$ we have that $\Ex{\norm{\nnoise_1-\nnoise_2}^d} < +\infty$. So taking expectations on \eqref{eq:cond_limit},
\begin{equation}
\lim_{\sigma \to 0} \frac{\Pr{C_2 \text{ is augmenting}}}{\sigma^d} = 2^{-d}\rho_d \Ex{\ltwo{\nnoise_1-\nnoise_2}^d} f_{\Delta_X}(0),
\end{equation}
follows by the dominated convergence theorem. Since $ f_{\Delta_X}(0) = \Ex{f_\idist(X_1)}$, \eqref{eq:aug_2_cycle_lim} follows.
\end{proof}

{\color{blue}
The result above implies that $\Pr{C_2 \text{ is augmenting}} = \tau \sigma^d + o(\sigma^d)$ for $\tau$ as in the right hand side of \eqref{lem:estim_aug_2_cycle}. This is the same constant appearing in \Cref{fig:error_rate} and the leading term in \Cref{thm:lss_up_fix_d}.
To recover the later proposition, we repeat the proof of \Cref{thm:lss_up_low_d} but use this sharper estimate for $\Pr{C_2 \text{ is augmenting}}$ on \eqref{eq:ex_bound_fix_d}.
}

We believe the upper bound from \Cref{thm:lss_up_fix_d} is sharp when $\sigma^d n = o(1)$ since in this regime $\Ex{\lssm}$ seems to be roughly the expected the mass of augmenting $2$-cycles which is,
\begin{equation}
 n^2 \Pr{C_2 \text{ is augmenting}}.   
\end{equation}
Since our sharp estimate from \Cref{lem:estim_aug_2_cycle} gives a sharp estimate on the latter quantity, it should give a sharp estimate on $\Ex{\lssm}$.

\begin{remark}
    We note that in the proofs above we only needed to use assumption \ref{cond:sub_gauss} from \modlds on $\ndist$. In particular, if we take $\ndist$ to be as in the examples shown in \Cref{fig:error_rate}, the results of this section still hold.
\end{remark}

\subsection{Upper bounds for high-dimensions}
\label{sec:up_bound_hd}

We now proceed to prove the upper bounds on the high-dimensional case from \Cref{mod:noniso_subgauss}. In this regime we also consider $\lssc$ as defined in \eqref{eq:lss_c}. The same strategy of bounding the mass of \textit{augmenting cycles} using \eqref{eq:det_bound_lss} also work for this estimator, but in this case a $t$-cycle $C_t$ is augmenting if and only if:
\begin{equation}
    \sum_{k=1}^t \langle (\Sigma_Z)^{-1} X_{i_k}, Y_{i_{k+1}} \rangle \geq \sum_{k=1}^t \langle (\Sigma_Z)^{-1} X_{i_k}, Y_{i_k} \rangle.
\end{equation}
The upper bounds on the probability of a given cycle being augmenting under \Cref{mod:noniso_subgauss} are given by: 
\begin{lemma}
\label{lem:bound_aug_2_cycle_subgaussian} Let $\Delta_1 := X_t - X_1$, $\Delta_i := X_{i-1} - X_i$ for $i = 2, \dots, t$. For any $t$-cycle $C_t$ we have that:
\begin{equation}
\label{eq:aug_cycle_bound_lss}
    \Pr{ C_t \text{ is augmenting for } \lss } \leq 2 \Ex{\exp\left(-\frac{1}{4 \kappa^2 K^2_Z} \frac{ \left( \sum\limits_{i=1}^t \ltwo{\Delta_i}^2 \right)^2} {\sum\limits_{i=1}^t \ltwo{ (\Sigma_Z)^{\frac{1}{2}} \Delta_i }^2}\right)},
\end{equation}
and,
\begin{equation}
\label{eq:aug_cycle_bound_lssc}
    \Pr{ C_t \text{ is augmenting for } \lssc } \leq 2 \Ex{\exp \left(-\frac{1}{4 \kappa^2 K^2_Z} \sum_{i=1}^t \ltwo{ (\Sigma_Z)^{-\frac{1}{2}} \Delta_i }^2\right)},
\end{equation}
where $\kappa > 0$ is an absolute constant and $K_Z > 0$ is any constant satisfying \ref{itm:kz} from \Cref{mod:noniso_subgauss}.

\end{lemma}

\begin{proof}
We prove again for $C_t = (1,2,\ldots,t)$ and the general case follows due to the exchangeability of $X_1, \dots, X_n$ and $Z_1, \dots, Z_n$. Following the proof of \Cref{lem:aug_up_bound_fix_d}, $C_t$ will be augmenting for $\lss$, if and only if:
\begin{equation}
\sum_{i=1}^t \langle \Delta_i, Z_i \rangle \geq \frac{1}{2} \left(\sum_{i=1}^t \ltwo{ \Delta_i }^2  \right).
\end{equation}
Conditioned on $\Delta_1,\dots,\Delta_t$, each $\langle \Delta_i, Z_i \rangle$ is a linear combination of independent zero-mean sub-Gaussian random variables with,
\begin{equation}
\begin{split}
    \snorm{ \binner{\Delta_i}{Z_i} }^2 &= \snorm{ \binner{ \Delta_i}{(\Sigma_Z)^{\frac{1}{2}} (\Sigma_Z)^{-\frac{1}{2}} Z_1}}^2 \\
    &= \snorm{ \left\langle  (\Sigma_Z)^{\frac{1}{2}} \Delta_i, (\Sigma_Z)^{-\frac{1}{2}} Z_1 \right\rangle }^2 \\ 
    &\leq \kappa \sum_{j=1}^d \left( \sigma_Z[j] \Delta_i[j] \right)^2 \snorm{ \nnoise_1[j]}^2 \\ 
    &\leq \kappa K^2_Z\ltwo{(\Sigma_Z)^{\frac{1}{2}} \Delta_i }^2,
\end{split}
\end{equation}
for some absolute constant $\kappa > 0$. It follows that (conditioned on $\Delta$),
\begin{equation}
    \snorm{ \sum_{i=1}^t \langle \Delta_i, Z_i \rangle }^2 \leq \kappa^2 K^2_Z \sum_{i=1}^t \ltwo{ (\Sigma_Z)^{\frac{1}{2}} \Delta_i }^2.
\end{equation}
Letting $\ubar{\dpos}$ be as in \Cref{lem:aug_up_bound_fix_d} we have that,
\begin{equation}
\begin{aligned}
\mathbb{P}\left(\sum_{i=1}^t \langle \Delta_i, Z_i \rangle \geq \frac{1}{2} \left(\sum\limits_{i=1}^t \ltwo{ \Delta_i }^2  \right) \Big\vert \ubar{\Delta} \right)
&\leq 2\exp \left(-\frac{1}{4 \kappa^2 K^2_Z} \frac{ \left(\sum\limits_{i=1}^t \ltwo{ \Delta_i }^2 \right)^2}{\sum\limits_{i=1}^t \ltwo{ (\Sigma_Z)^{\frac{1}{2}} \Delta_i }^2} \right).
\end{aligned}
\end{equation}
For $\lssc$, $C_t$ will be augmenting if and only if,
\begin{equation}
\label{eq:aug_cycle_lssc}
\sum_{i=1}^t \left\langle (\Sigma_Z)^{-1}\Delta_i, Z_i \right\rangle \geq \frac{1}{2} \left(\sum_{i=1}^t \ltwo{ (\Sigma_Z)^{-\frac{1}{2}} \Delta_i }^2  \right).
\end{equation}
In this case, conditioned on $\Delta_i$,
\begin{equation}
\begin{split}
    \snorm{ \binner{(\Sigma_Z)^{-1} \Delta_i}{Z_i}}^2 &= \snorm{ \binner{(\Sigma_Z)^{-\frac{1}{2}} \Delta_i}{ (\Sigma_Z)^{-\frac{1}{2}} Z_1 } }^2 \\
    &\leq \kappa \sum_{j=1}^d \left((\Sigma_Z)^{-\frac{1}{2}}[j] \Delta_i[j]\right)^2 \snorm{ \nnoise_1[j] }^2 \\
    &\leq \kappa K^2_Z \ltwo{(\Sigma_Z)^{-\frac{1}{2}} \Delta_i}^2,
\end{split}
\end{equation}
and so conditioned on $\Delta_1, \dots, \Delta_t$,
\begin{equation}
    \snorm{ \sum_{i=1}^t \langle (\Sigma_Z)^{-1} \Delta_i, Z_i \rangle }^2 \leq \kappa^2 K^2_Z \sum_{i=1}^t \ltwo{(\Sigma_Z)^{-\frac{1}{2}} \Delta_i}^2.
\end{equation}
Thus,
\begin{equation}
\begin{aligned}
\mathbb{P}\left(\sum_{i=1}^t \langle (\Sigma_Z)^{-1}\Delta_i, Z_i \rangle \geq \frac{1}{2} \left(\sum_{i=1}^t \ltwo{ (\Sigma_Z)^{-\frac{1}{2}} \Delta_i }^2  \right)  |  \ubar{\Delta} \right) &\leq
2 \exp\left(-\frac{1}{4 \kappa^2 K^2_Z} \sum_{i=1}^t \ltwo{ (\Sigma_Z)^{-\frac{1}{2}} \Delta_i }^2\right).
\end{aligned}
\end{equation}
Taking expected values with respect to $\Delta_1, \dots, \Delta_n$ we have the result.
\end{proof}

In order to estimate the expected values on the right hand side of \eqref{eq:aug_cycle_bound_lss} and \eqref{eq:aug_cycle_bound_lssc} we use the concentration of the norm of $\Delta_1, \dots, \Delta_t$. We do so with the following consequence of Hanson-Wright Inequality \cite{rudelson2013hanson}:
\begin{lemma}
\label{lem:conc_quad_abs}
Let $U_1, \dots, U_t$, $t \geq 2$, be i.i.d. random vectors in $\R^d$ with zero-mean, unit variance, independent sub-Gaussian coordinates and let $D$ be a diagonal matrix with non-negative diagonal entries. We have that for any $\epsilon > 0$,

\begin{equation}
    \Pr{ \left| \sum_{i=1}^t \| D^{\frac{1}{2}} \Xi_i \|^2_2 - 2t\tr{D} \right| \geq \varepsilon 2t\tr{D} } \leq e^{-c_1\left(\frac{\varepsilon}{K^2}\right) 2t \frac{\tr{D}}{\lVert D \rVert_{op}}},
\end{equation}
where $\Xi_1 := U_t - U_1$, $\Xi_i := U_{i-1} - U_i$ for $i = 2, \dots, t$, $\max_{j \in [d]} \snorm{U_1[j]} \leq K$ and $c_1\left(\frac{\varepsilon}{K^2}\right) > 0$ is a positive constant that depends on $\frac{\varepsilon}{K^2}$. 

\end{lemma}
\begin{proof} Let $U \in \R^{td}$ be a concatenation of $U_1, U_2, \dots, U_t$, \ie
\begin{equation}
    U[d(i-1)+j] = U_i[j], \text{ for } i \in [t] \text{ and } j \in [d].
\end{equation}
We first consider the case $t \geq 3$. Let $L^{C_t}$ and $E^{C_t}$ be the Laplacian and incidence matrix respectively, of the cycle graph on $t$ vertices. Taking $A = \left( L^{C_t} \otimes D \right)$ and $B = \left( E^{C_t} \otimes D^{\frac{1}{2}} \right)$ we have that $A = B B^T$. In particular,
\begin{equation}
\begin{split}
 \sum_{i=1}^t \| D^{\frac{1}{2}}\Xi \|^2_2 &= \sum_{i=1}^t \binner{D^{\frac{1}{2}}\Xi_i}{D^{\frac{1}{2}}\Xi_i} \\
 &= \sum_{i=1}^t \sum_{j=1}^d \left(D^{\frac{1}{2}}[j,j]\Xi_i[j]\right)^2\\
 &= \binner{B^TU}{B^TU} \\
 &= \binner{AU}{U}. 
\end{split}
\end{equation}
Since the eigenvalues of $L^{C_t}$ are given by $4\sin^2\left( \frac{\pi k}{t} \right), ~ k = 0, \dots, t-1$ \cite{spec_graph_theory_2018}, we have that,
\begin{equation}
    \opnorm{ A } = \opnorm{ L^{C_t} } \opnorm{ D } = 4 \text{sin}^2\left( \frac{\pi \lc \frac{t}{2} \rc }{t} \right) \opnorm{ D } \leq 4 \opnorm{ D }.
\end{equation}
Also,
\begin{equation}
    \ltwo{A}^2 = \ltwo{L^{C_t}}^2 \ltwo{ D }^2 = 6t \tr{D^2} \leq 6t \opnorm{ D } \tr{D},
\end{equation}
and,
\begin{equation}
    \tr{A} = 2t\tr{D}.
\end{equation}
Since $U$ has independent, zero-mean and unit variance coordinates we have that $\Ex{\binner{A U}{U}} = \tr{A}$. It follows that using Hanson-Wright inequality,

\begin{equation}
\begin{aligned}
    \Pr{\left|\binner{A U}{U} -  \tr{A}\right| \geq \varepsilon \tr{A}} &\leq 2 \exp \left(-c ~ \text{min}\left( \frac{\varepsilon^2 \tr{A}^2}{K^4 \ltwo{A}^2}, \frac{\varepsilon \tr{A}}{K^2 \opnorm{ A }}  \right) \right) \\
    &\leq 2\exp \left(-c ~ \text{min}\left( \frac{\varepsilon^2 (2 t \tr{D})^2}{K^4 6t \tr{D} \opnorm{ D }}, \frac{\varepsilon 2 t \tr{D}}{4 K^2 \opnorm{ D }}  \right) \right) \\
    &= 2\exp \left(- 2 t \frac{\tr{D}}{\opnorm{ D }} ~ c ~ \text{min}\left( \frac{\varepsilon^2}{3 K^4} , \frac{\varepsilon}{4 K^2} \right) \right),
\end{aligned}
\end{equation}
for some absolute constant $c > 0$.
Following a similar reasoning, if $t=2$, letting $A = 2 \left( L^{P_2} \otimes D \right)$, where $L^{P_2}$ is the Laplacian of the path graph on $2$ vertices, we have: 
\begin{equation}
\begin{split}
    \sum_{i=1}^2 \| D^{\frac{1}{2}}\Xi_i \|^2_2 &= 2 \| D^{\frac{1}{2}}\Xi_1 \|^2_2 \\
    &= \binner{AU}{U}.
\end{split}
\end{equation}
Since the eigenvalues of $L^{P_t}$ are $0$ and $2$  \cite{spec_graph_theory_2018} we have,
\begin{equation}
    \opnorm{ A } = 4 \opnorm{ D },
\end{equation}
\begin{equation}
    \ltwo{A}^2 = 16 \tr{D^2} \leq 16 \opnorm{ D } \tr{D},
\end{equation}
and,
\begin{equation}
    \tr{A} = 4\tr{D}.
\end{equation}
It follows that,
\begin{equation}
\begin{aligned}
    \Pr{\left|\binner{A U}{U} -  \tr{A}\right| \geq \varepsilon \tr{A}} 
    &\leq 2 \exp \left(-c ~ \text{min}\left( \frac{\varepsilon^2 \tr{A}^2}{K^4 \ltwo{A}^2}, \frac{\varepsilon \tr{A}}{K^2 \opnorm{ A }}  \right) \right) \\
    &\leq 2\exp \left(-c ~ \text{min}\left( \frac{\varepsilon^2 (4 \tr{D})^2}{16 K^4 \tr{D} \opnorm{ D }}, \frac{\varepsilon 4 \tr{D}}{4 K^2 \opnorm{ D }}  \right) \right) \\
    &\leq 2\exp \left(-4\frac{\tr{D}}{\opnorm{ D }} ~ c ~ \text{min}\left( \frac{\varepsilon^2}{4 K^4} , \frac{\varepsilon}{4 K^2} \right) \right).
\end{aligned}
\end{equation}
The result follows by taking $c_1\left( \frac{\varepsilon}{K^2} \right) = c ~ \text{min}\left( \frac{\varepsilon^2}{4K^4} , \frac{\varepsilon}{4 K^2} \right)$.
\end{proof}

We now prove \Cref{thm:high_dim_lss} taking $b_{\text{HD}} := 2 \kappa^2$.

\begin{proof}[Proof of \Cref{thm:high_dim_lss}]
\label{proof:high_dim_lss}

Given $\varepsilon > 0$ and $t \geq 2$, we have that,
\begin{equation}
\begin{aligned}
\Ex { \exp\left(-\frac{1}{4 \kappa^2 K^2_Z} \frac{ \left( \sum\limits_{i=1}^t \ltwo{ \Delta_i }^2 \right)^2} {\sum\limits_{i=1}^t \ltwo{ (\Sigma_Z)^{\frac{1}{2}} \Delta_i }^2} \right) }
&\leq \expo{ -\frac{1}{4 \kappa^2 K^2_Z}\frac{(1-\varepsilon)^2}{(1+\varepsilon)}\frac{(2 t \tr{\Sigma_X})^2}{(2 t \tr{\Sigma_Z \Sigma_X)}}} \\
&+ \Pr{ \sum_{i=1}^t \ltwo{\Delta_i}^2 \leq (1-\varepsilon) 2 t \tr{\Sigma_X} }\\
&+ \Pr{ \sum_{i=1}^t \ltwo{(\Sigma_Z)^{\frac{1}{2}}\dpos_i}^2 \geq (1+\varepsilon) 2 t \tr{\Sigma_Z \Sigma_X} }\\
&\leq \exp \left( -\frac{1}{4 \kappa^2 K^2_Z}\frac{(1-\varepsilon)^2}{(1+\varepsilon)}\frac{(2 t \tr{\Sigma_X})^2}{(2 t \tr{\Sigma_Z \Sigma_X)}} \right) \\
&+ \exp \left(-2 t \frac{\tr{\Sigma_X}}{\opnorm{\Sigma_X}} c_1\left(\frac{\varepsilon}{(K_X)^2}\right) \right) \\
&+ \exp \left( -2 t \frac{\tr{\Sigma_Z \Sigma_X}}{\opnorm{ \Sigma_Z \Sigma_X }} c_1\left(\frac{\varepsilon}{(K_X)^2}\right) \right),
\end{aligned}
\end{equation}
where the last inequality follows from Lemma \ref{lem:conc_quad_abs} taking $U_1 = \tilde{X}_1$ with $D = \Sigma_X$ and with $D = \Sigma_Z \Sigma_X$ respectively. Using the same reasoning as in \eqref{eq:nt_bound} it follows that,

\begin{align}
\Ex{\lssm}
&\leq\sum_{t=2}^n n^t \Pr{ (1,\dots,t) \text{ is augmenting for } \lss} \\
&\leq \sum_{t=2}^n 2 \exp(t\log n) \Ex{ \exp \left(-\frac{1}{4 \kappa^2 K^2_Z} \frac{ \left( \sum\limits_{i=1}^t \ltwo{\Delta_i}^2 \right)^2} {\sum\limits_{i=1}^t \ltwo{ (\Sigma_Z)^{\frac{1}{2}} \Delta_i }^2} \right) } \\
&\leq 2\sum\limits_{t=2}^n \exp \left(t \log n \left(1 -\frac{1}{2 \kappa^2 K^2_Z \log n}\frac{(1-\varepsilon)^2}{(1+\varepsilon)}\frac{\tr{\Sigma_X}}{\frac{\tr{\Sigma_Z \Sigma_X}}{\tr{\Sigma_X}}}\right)\right) \\
&+ 2\sum_{t=2}^n \exp \left( t \log n \left(1 - 2 \frac{\tr{\Sigma_X}}{\log n \opnorm{ \Sigma_X }} c_1\left(\frac{\varepsilon}{(K_X)^2}\right)\right) \right) \\
&+ 2\sum_{t=2}^n \exp \left(t \log n\left(1 - 2 \frac{\tr{\Sigma_Z \Sigma_X}}{\log n\opnorm{ \Sigma_Z \Sigma_X }} c_1\left(\frac{\varepsilon}{(K_X)^2}\right)\right) \right) \\
&\leq 6 \sum_{t=2}^{+\infty} e^{-t\omega(1)} = o(1),
\end{align} 
where the last equality is a consequence of our hypothesis taking $b_{\mathrm{HD}} = 2 \kappa^2$. It follows by Markov's inequality that we have perfect recovery.
\end{proof}
The proof of \Cref{thm:high_dim_lssc} is analogous.
\begin{proof}[Proof of \Cref{thm:high_dim_lssc}]
\label{proof:high_dim_lssc}

Note that,
\begin{equation}
\begin{aligned}
    \Ex{\sum_{i=1}^t \lVert (\Sigma_Z)^{-\frac{1}{2}}\Delta_i \rVert^2_2} &= 2 t \tr{(\Sigma_Z)^{-1} \Sigma_X}.
\end{aligned}
\end{equation}

Given $\varepsilon > 0$, 

\begin{equation}
\begin{aligned}
\label{eq:bound_mgf_conc}
    \Ex{\exp \left(-\frac{1}{4 \kappa^2 K^2_Z} \sum_{i=1}^t \ltwo{ (\Sigma_Z)^{-\frac{1}{2}} \Delta_i }^2\right)} &\leq \exp \left(-\frac{(1-\varepsilon)}{4 \kappa^2 K^2_Z}  2 t \tr{(\Sigma_Z)^{-1} \Sigma_X} \right) 
    ~ \\ 
    &+ 
    \Pr{ \sum_{i=1}^t \ltwo{(\Sigma_Z)^{-\frac{1}{2}}\dpos_i}^2 \leq  (1-\varepsilon) 2 t \tr{(\Sigma_Z)^{-1} \Sigma_X}  }.
\end{aligned}
\end{equation}

We can bound the last term using \Cref{lem:conc_quad_abs} once again with $U_1 = \tilde{X}_1$ and $D = (\Sigma_Z)^{-1}\Sigma_X$ to estimate the last probability above. We then have,

\begin{align}
\Ex{\lsscm} &\leq 2 \sum_{t=2}^n \exp \left( t\log n \right) \Ex{ \exp \left(-\frac{1}{4 \kappa^2 K^2_Z} \sum_{i=1}^t \lVert(\Sigma_Z)^{-\frac{1}{2}}\dpos_i \rVert^2_2 \right) } \\ 
& \leq 2 \sum_{t=2}^n \exp \left(t \log n\left( 1 -\frac{(1-\varepsilon)}{2 \kappa^2 K^2_Z \log n} \tr{(\Sigma_Z)^{-1}\Sigma_X} \right) \right) \\
&+ 2 \sum_{t=2}^n  \exp \left( t \log n\left( 1 - 2 c_1\left(\frac{\varepsilon}{K_X^2}\right) \frac{\tr{(\Sigma_Z)^{-1}\Sigma_X}}{\log n \opnorm{ (\Sigma_Z)^{-1}\Sigma_X }} \right) \right) \\
&\leq 4 \sum_{t=2}^{+\infty} e^{-t\omega(1)} = o(1),
\end{align} 
where the last equality is again a consequence of our hypothesis taking $b_{\mathrm{HD}} = 2 \kappa^2$.  Once again we have perfect recovery following Markov's inequality.
\end{proof}

\section{Future work}
\label{sec:conj_future_work}

Besides the open problems already posed in \cite{geo_match_2022}, we propose the following directions for future work:

\begin{direction*}
\bluenew{Although the minimax lower bounds established in this work hold under mild assumptions on the tail of the noise distribution, we were only able to show that they are optimal for $\lss$ and $\gdpi$ by either leaving $d$ fixed or making the assumption that the noise is sub-Gaussian. This motivates the following problem: rigorously establish general conditions under which $\lss$ and $\gdpi$ are expected to be sub-optimal. Two specified cases being:
\begin{enumerate}
    \item In the high-dimensional regime, provide sufficient conditions under which the conjectured signal-to-noise ratios in \Cref{thm:high_dim_lss} and \Cref{thm:high_dim_lssc} are in fact signal-to-noise ratios, and provide an example in which \eqref{eq:lssc_beats_lss} implies \textit{perfect recovery} for $\lssc$ but only \textit{partial recovery} for $\lss$.
    \item Prove or disprove \Cref{conj:mle_beats_lss}.
\end{enumerate}
}
\end{direction*}

\begin{direction*}
 Prove the following conjecture of independent interest:
\begin{conjecture}
\label{conj:size_largest_matching} Given $X_1, \dots, X_n \iid \idist_n$, let $M_r$ be the largest matching in $G(\{X_1, \dots, X_n\},r,d,\enorm)$, where we allow both $r$ and $d$ to scale with $n$. Assume,
\begin{equation}
   \Ex{E(G(n,r,d,\enorm))} = \binom{n}{2} \Pr{\| X_1 - X_2 \|_2^2 < r^2} \tolim{n} +\infty.
\end{equation}
Then:
\begin{equation}
    \Ex{|M_r|} =  \Theta \left( \Ex{E(G(n,r,d,\enorm))} \wedge n \right),
\end{equation}
and moreover the same order holds with high probability for $|M_r|$.
\end{conjecture}
That is, the size of the largest matching in a random geometric graph is of the order of the minimum between the expected number of edges of $G(n,r,d,\enorm)$ and $n$. Proving the high-probability version of this conjecture in the regime where $d = \Theta(\log n)$ would allow us to partially recover the lower bound for the logarithmic regime (\ie $d = \Theta(\log n)$) from \cite{geo_match_2022}, following a more general framework, as we outline in \Cref{a:log_reg}.
\end{direction*}

\appendix

\begin{appendices}

\section{A technical result on estimation in groups}
\label{a:estim_group}

We prove the equality in \eqref{eq:minimax_bayes_id} under a more general setting.

\begin{lemma}
\label{lem:minimax_bayes}
    Let $\{P_g \}_{g \in G}$ be a family of probability measures on some measure space $\left(\mathcal{Z}, \mathcal{A}\right)$, indexed by $G$ a finite group. Assume that there exists a family of (measurable) mappings $\{ T_g: \mathcal{Z} \to \mathcal{Z}; ~ g \in G \}$ such that for any $g,h \in G$,
    \begin{equation}
    \label{eq:inv_map}
        Z \sim P_g \mbox{ implies } T_h(Z) \sim P_{gh}.
    \end{equation}
    Let $d_G:G \times G \to \R_{+}$ be a distance metric which is invariant under right multiplication by a group element, \textit{i.e.}, for any $ g,g',h \in G$,
    \begin{equation}
        d_G(g,g') = d_G(gh,g'h).
    \end{equation}
    Let $g^* \sim \mathrm{Unif}(G)$ be a random group element and $Z \sim P_{g^*}$. Then,
    \begin{equation}
        \inf_{\hat{g}} \max_{g \in G} \Ex{d_G(\hat{g}(Z),g^*)|g^* = g} = \inf_{\hat{g}}\mathbb{E}\left[d_G(\hat{g}(Z),g^*)\right],
    \end{equation}
where the infimum is taken over all (possibly non-deterministic) estimators $\hat{g}:\mathcal{Z} \to G$ of $g^*$.
\end{lemma}

\begin{proof}
    It is always the case that,
    \begin{equation}
        \inf_{\hat{g}} \max_{g \in G} \Ex{d_G(\hat{g}(Z),g^*)|g^* = g} \geq \inf_{\hat{g}}\mathbb{E}\left[d_G(\hat{g}(Z),g^*)\right].
    \end{equation}
As for the other direction, let $\hat{g}$ be some (possibly non-deterministic) estimator of $g^*$. Letting $\tilde{h} \sim \text{Unif}(G)$ be sampled independently from $Z$ and $g^*$, set $\nnoise = T_{\tilde{h}}(Z)$. Note that $\nnoise \sim P_{g^*\tilde{h}}$ by \eqref{eq:inv_map}. Define a new non-deterministic estimator:
\begin{equation}
    \tilde{g}(Z) := \hat{g}(\nnoise) \tilde{h}^{-1}. 
\end{equation}
Due to the invariance of $d_G$,
\begin{equation}
    d_G(\tilde{g}(Z),g^*) = d_G(\hat{g}(\nnoise),g^*\tilde{h}).
\end{equation}
Also, conditioned on $g^*$, $g^*\tilde{h} \sim \text{Unif}(G)$. It follows that the conditional distribution of $d_G(\hat{g}(\nnoise),g^*\tilde{h})$ given $g^*$ is equal to the (unconditional) distribution of $d_G(\hat{g}(Z),g^*)$. So that for any $g \in G$,
\begin{equation}
    \Ex{d_G(\tilde{g}(Z),g^*)|g^*=g} = \Ex{d_G(\hat{g}(\nnoise),g^*\tilde{h})|g^*=g} = \Ex{d_G(\hat{g}(Z),g^*)},
\end{equation}
and so,
\begin{equation}
    \max_{g \in G} \Ex{d_G(\tilde{g}(Z),g^*)|g^*=g} \leq \Ex{d_G(\hat{g}(Z),g^*)}.
\end{equation}
From this it follows that,
\begin{equation}
    \inf_{\hat{g}} \max_{g \in G} \Ex{d_G(\hat{g}(Z),g^*)|g^* = g} \leq \inf_{\hat{g}}\mathbb{E}\left[d_G(\hat{g}(Z),g^*)\right].
\end{equation}

\end{proof}
In the particular setting of \Cref{lem:gen_inf_theo_lb} we have that \Cref{lem:minimax_bayes} holds with $G = S_n$, $\mathcal{Z} = (\R^d)^{n}$, $P_{\pi}$ the law of $(Y_1, \dots, Y_n)$ where,
\begin{equation}
    Y_i := x_{\pi(i)} + \sigma \nnoise_i, ~ i \in [n],
\end{equation}
for $\nnoise_1, \dots, \nnoise_n \iid Q$, $d_G$ the Hamming distance and,
\begin{equation}
    T_{\alpha}(Y_1,\dots,Y_{n}) := (Y_{\alpha(1)},\dots,Y_{\alpha(n)}).
\end{equation}

\pagebreak

\section{Comparisons with other estimators}
\label{a:other_estim}

\subsection{Maximum likelihood estimator}

We provide a proof of \Cref{lem:mle_beats_lss_2_cyc}.

\begin{lemma}
    Assuming $\tpi = \mathrm{Id}$ and $\mathcal{Q}$ is absolutely continuous, we have,
    \begin{equation}
        \begin{split}
            \Pr{W^{\mathcal{Q},\sigma}_{1,2} + W^{\mathcal{Q},\sigma}_{2,1} > W^{\mathcal{Q},\sigma}_{1,1} + W^{\mathcal{Q},\sigma}_{2,2}} \leq \Pr{\ltwo{X_1-Y_2}^2 + \ltwo{X_2-Y_1}^2 < \ltwo{X_1-Y_1}^2 + \ltwo{X_2-Y_2}^2}.
        \end{split}
    \end{equation}
In other words, the probability of $(1 ~ 2)$ being an augmenting cycle for $\mle$ is never larger than that of $(1 ~ 2)$ being an augmenting cycle for $\lss$.
\end{lemma}
\begin{proof}
    Let $\pi_0 = \mathrm{Id}$ and $\pi_1 = (1 ~ 2)$. Due to the sequence of initial positions and noise being i.i.d., we have that,
\begin{equation}
    \P_{\pi_0}\left(W^{\mathcal{Q},\sigma}_{1,2} + W^{\mathcal{Q},\sigma}_{2,1} > W^{\mathcal{Q},\sigma}_{1,1} + W^{\mathcal{Q},\sigma}_{2,2}\right) = \P_{\pi_1}\left(W^{\mathcal{Q},\sigma}_{1,1} + W^{\mathcal{Q},\sigma}_{2,2} > W^{\mathcal{Q},\sigma}_{1,2} + W^{\mathcal{Q},\sigma}_{2,1}\right),
\end{equation}

where $\P_{\pi_i}$ indicates a probability assuming $\pi_i$ is the correct assignment. Denote this quantity by $\alpha$. If we had,

\begin{equation}
    \label{eq:mle_smaller_level}
        \begin{split}
            \P_{\pi_0}\left(W^{\mathcal{Q},\sigma}_{1,2} + W^{\mathcal{Q},\sigma}_{2,1} > W^{\mathcal{Q},\sigma}_{1,1} + W^{\mathcal{Q},\sigma}_{2,2}\right) > \P_{\pi_0}\left(\ltwo{X_1-Y_2}^2 + \ltwo{X_2-Y_1}^2 < \ltwo{X_1-Y_1}^2 + \ltwo{X_2-Y_2}^2\right),
        \end{split}
\end{equation}

then we could construct an $\alpha$-level test for $\pi_0$ against $\pi_1$ by looking at whether $\ltwo{X_1-Y_2}^2 + \ltwo{X_2-Y_1}^2 < \ltwo{X_1-Y_1}^2 + \ltwo{X_2-Y_2}^2$ happens (reject $\pi_0$) or not (fail to reject $\pi_0$). However, the analogous test obtained by looking at $W^{\mathcal{Q},\sigma}_{i,j}$ is also an $\alpha$-level test. Hence by the Neyman-Pearson lemma \cite{neyman1933}, no other test can be strictly more powerful than it. But this contradicts,
\begin{equation}
        \begin{split}
            \P_{\pi_1}\left(W^{\mathcal{Q},\sigma}_{1,2} + W^{\mathcal{Q},\sigma}_{2,1} > W^{\mathcal{Q},\sigma}_{1,1} + W^{\mathcal{Q},\sigma}_{2,2}\right) < \P_{\pi_1}\left(\ltwo{X_1-Y_2}^2 + \ltwo{X_2-Y_1}^2 < \ltwo{X_1-Y_1}^2 + \ltwo{X_2-Y_2}^2\right).
        \end{split}
\end{equation}
which follows from \Cref{eq:mle_smaller_level} and the exchangeability of the distribution. Hence \Cref{eq:mle_smaller_level} cannot hold and we have the desired result.
\end{proof}

\subsection{Greedy distance}

Following the discussion in the Appendix of \cite{geo_match_2022} for the greedy distance algorithm, which matches each point to its nearest neighbor (with an error declared if the neighbor is not available anymore), we have that the number of mistakes made by such algorithm is bounded as:
\begin{equation}
    \label{eq:hm_gd}
    \Ex{\hamdist{\gdpi}{\tpi}} \leq n^2 \Pr{\ltwo{Y_1-X_1} \geq \ltwo{Y_2-X_1}}.
\end{equation}

Under the weaker assumption that $\Ex{\ltwo{\nnoise_1}^d} < +\infty$, we have the following analogous estimate to \Cref{lem:estim_aug_2_cycle}:
\begin{lemma}
    Under the assumptions of \Cref{lem:estim_aug_2_cycle} with $d$ fixed, we have,
    \begin{equation}
    \lim_{\sigma \to 0^+} \frac{\Pr{\ltwo{Y_1-X_1} \geq \ltwo{Y_2-X_1}}}{\sigma^d} = \rho_d \Ex{f_\idist(X_1)} \Ex{\ltwo{\nnoise_1}^d}  
    \end{equation}
\end{lemma}
\begin{proof}
    The proof is the same \Cref{lem:estim_aug_2_cycle}, with the only difference that now,
    \begin{equation}
        \begin{split}
            \Pr{\ltwo{Y_1-X_1} \geq \ltwo{Y_2-X_1}} = \Pr{\ltwo{Z_1} \geq \ltwo{Z_2 - \Delta_X}, \ltwo{Z_1} > 0} 
        \end{split}
    \end{equation}
hence, the pointwise limit of the conditional probability in \eqref{eq:cond_limit} becomes:
\begin{equation}
    \lim_{\sigma \to 0} \frac{\rho_d\ltwo{Z_1}^d}{\sigma^d}\frac{\Pr{\ltwo{Z_1} \geq \ltwo{Z_2 - \Delta_X}  | Z_1,Z_2}}{\rho_d\ltwo{Z_1}^d}\indic{\ltwo{Z_1} > 0} = \rho_d \ltwo{\nnoise_1}^d f_{\Delta_X}(0)\indic{\ltwo{Z_1} > 0}  \text{ a.s.}
\end{equation}
\end{proof}

This implies that the greedy distance algorithm satisfies $\Ex{\hamdist{\gdpi}{\tpi}} \lesssim n^2 \sigma^d$ hence it is also minimax optimal under \modlds, but with a weaker requirement on the tail of the noise. Moreover, by the triangle inequality, we have that,
\begin{equation}
 2^{-d}\Ex{\ltwo{\tilde{Z}_1-\tilde{Z}_2}^d} \leq \Ex{\ltwo{\tilde{Z}_1}^d},   
\end{equation}
hence the upper bounds for the error of LSS is strictly better, when the mass of augmenting 2-cycles dominates the error.
Overall this shows that the equivalence between the rates for LSS and the greedy matching algorithm for low dimensions also holds under greater assumptions, extending the observation in \cite{geo_match_2022} for the Gaussian model.
One crucial thing to note however, is that in the expression:
\begin{equation}
 \rho_d \Ex{f_\idist(X_1)} \Ex{\ltwo{\nnoise_1}^d},   
\end{equation}
the dominant terms in terms of the scaling will be $\rho_d$ and $\Ex{\ltwo{\nnoise_1}^d}$. The case in which $\nnoise$ is sub-Gaussian corresponds precisely to both being comparable having order $\Theta(d^{\frac{d}{2}})$. This comparability in the sub-Gaussian regime is what allows the error rate of the certain models with slowly growing dimension to still be nearly sharp as discussed in \Cref{sec:examples}. It is easy to see however, that a sub-Gaussian tail is essentially the best we can hope for, as weakening this assumption would imply that $\Ex{\ltwo{\nnoise_1}^d}$ grows faster than $O(d^{\frac{d}{2}})$. Hence as soon as $d$ starts growing with $n$, even if slowly, we expect both the greedy distance and the LSS estimators to not give sharp rates if the noise is worse than sub-Gaussian.

For $d$ growing in a way which is $\Omega(\log(n))$, we already know from \cite{geo_match_2022} that sufficient conditions for greedy distance derived from \eqref{eq:hm_gd} can be strictly worse: $\sigma^2 = O\left(\sqrt{\frac{d}{\log(n)}}\right)$ against $\sigma^2 = O(\frac{d}{\log(n)})$ for the MLE (which is the LSS for the Gaussian model).

\pagebreak

\section{Negative results in probability}
\label{a:high_prob_lb}

In this section we show how one can use the same strategy from \Cref{sec:lower_bounds_rggs} to derive high-probability lower bounds for $\lssm$ that recover the results from \cite{geo_match_2022} up to $n^{o(1)}$ factors when $d = o(\log n)$, and also how it can be used to establish negative results that complement the ones from \cite{wang2022random}.

\subsection{high-probability lower bounds on LSS}
\label{sec:lss_lb}

Our goal in this section is to prove the following high-probability analogue of \Cref{thm:lb_low_d} which partially recovers the lower bound from \cite{geo_match_2022} when $d = o(\log n)$, under greater generality:
\begin{theorem}[high-probability lower bound on $\lssm$]
\label{thm:hp_lss_lb}
Consider a model as in \eqref{eq:final_pos} with $X_1 \sim \idist_n$ following assumption \ref{cond:i_pos_low_dim} from \Cref{mod:low_dim} for each $n$ for some $R_d, \enorm, \gamma$ and $\beta$, with the last two constants not depending on $n$. Assume also that $\noise_1 = \sigma \nnoise_1$, where $\sigma > 0$, $\nnoise_1, \nnoise_2 \iid \ndist_n$ a sequence of distributions in $\R^d$ such that,
\begin{equation}
\label{eq:q_lb_cond}
     q_{R_d}(\ndist) = q_{R_d}(\ndist_n,\enorm) := \inf_{\norm{v}=1} \Pr{\binner{\nnoise_1 - \nnoise_2}{v} \geq R_d} \geq e^{-\beta d},
\end{equation}
for the same $R_d$ and $\beta$ as in the assumption on $\idist_n$. If $e^{-10 \beta d}\sigma^d n^2 \tolim{n} +\infty$ then with high probability,
\begin{equation}
\label{eq:lssm_lb}
       \lssm \geq \frac{\gamma^2}{64} e^{-7\beta d} \left( (n^2 \sigma^d) \wedge n \right).
\end{equation}
\end{theorem}

For brevity we omit the dependence on $n$ of $\idist_n$ and $\ndist_n$.
We already know that if we assume $d$ and $\idist$ to be fixed with the latter having a bounded continuous density, assumption \ref{cond:i_pos_low_dim} is verified. One can also verify that if $\ndist$ is fixed with a bounded continuous density then $q_{s^*}(\ndist) > 0$ for some $s^* > 0$, and thus \eqref{eq:q_lb_cond} also holds, thus giving us a high-probability lower bound of order $(n^2 \sigma^d) \wedge n$ for $\lssm$. Similarly, if $\idist$ and $\ndist$ follow the Gaussian model and $d = o(\log n)$ one can verify that the assumptions of \Cref{thm:hp_lss_lb} also hold and obtain a high-probability lower bound of order $(n^{(2-o(1))} \sigma^d) \wedge n^{(1-o(1))}$ for $\lssm$.

As noted previously, when dealing with $\lssm$, augmenting $2$-cycles are of particular interest as they seem to dominate the total mass of augmenting $t$-cycles in the case of fixed dimension when $\sigma^d = o\left(\frac{1}{n}\right)$, as seen in \cite{geo_match_2022} and as suggested by the proof of \Cref{thm:lss_up_low_d}. As in \cite{geo_match_2022}, we let,
\begin{equation}
 \gaug = \gaug\left(n; X_1, \dots, X_n, Y_1, \dots, Y_n \right),   
\end{equation}
be the graph of augmenting $2$-cycles characterized by:

\begin{equation}
    V(\gaug) = [n] \text{ and } E(\gaug) = \{ \{i,j \} \subseteq [n], \, (i,j) \text{ is an augmenting cycle}\}.
\end{equation}

Proving the existence of large matchings in $\gaug$ can be used to prove lower bounds on $\Ex{\lssm}$ due to the following result from \cite{geo_match_2022} (Proposition 4.1):

\begin{lemma}
\label{lem:gaug_lb}
Let $M^*$ be a largest matching in $\gaug$. We have,
\begin{equation}
    \lssm \geq |M^*|.
\end{equation}
\end{lemma}

In order to show the existence of a large matching on $\gaug$, instead of using an enhanced second moment argument as in \cite{geo_match_2022}, we use the fact that a fraction of edges from the largest matching in $G(n,d,r,\enorm)$ will also be augmenting cycles thus enabling us to move from a large matching in the random geometric graph induced by the initial positions to a large matching on the graph of augmenting $2$-cycles. Since we are dealing with $\lss$, in this subsection we assume that $\tpi = \text{Id}$ as in \Cref{sec:up_bounds}. Letting $\nnoise_1, \nnoise_2 \iid \ndist$, set:
\begin{equation}
\label{eq:q_def}
    q_s(\ndist) = q_s(\ndist, \enorm) := \inf_{\norm{v}=1} \Pr{\binner{\nnoise_1 - \nnoise_2}{v} \geq s}.
\end{equation}
We have the following high-probability analogue of \Cref{lem:gen_inf_theo_lb} for $\lssm$:
\begin{lemma}
\label{lem:lss_gen_low_bound}
Let $x_1, \dots, x_n \in \R^d$ be distinct and set:
\begin{equation}
    Y_i := x_i + \sigma \nnoise_i, ~ i \in [n],
\end{equation}
where $\nnoise_1, \dots, \nnoise_n \iid \ndist$ and $\sigma > 0$. Let $q_{s}(\ndist)$ be as in \eqref{eq:q_def} above and $M_r = M_r(\{x_1,\dots,x_n\})$ the largest matching in $G(\{x_1, \dots, x_n \}, r, d, \ltwo{\, \cdot \,})$. If $|M_r|q_{\frac{r}{\sigma}}(\ndist) \tolim{n} +\infty$, then,
\begin{equation}
    \lssm \geq \frac{|M_r|q_{\frac{r}{\sigma}}(\ndist)}{2}
\end{equation}
with high probability.
\end{lemma}

\begin{proof}

Let $\gaug = \gaug\left(n; x_1, \dots, x_n, Y_1, \dots, Y_n \right)$ be the graph of augmenting $2$-cycles and let $M^*$ denote its largest matching. Given $\{ i, j \} \in M_r$,

\begin{equation}
\begin{split}
    \Pr{\{i,j\} \text{ is an edge on } \gaug } &= \Pr{ \binner{\nnoise_i - \nnoise_j}{\frac{x_i - x_j}{\norm{x_i - x_j}}} \geq \frac{\|x_i - x_j\|}{\sigma}} \\
    &\geq \Pr{ \binner{\nnoise_i - \nnoise_j}{\frac{x_i - x_j}{\norm{x_i - x_j}}} \geq \frac{r}{\sigma}} \\
    &\geq q_{\frac{r}{\sigma}}(\ndist).
\end{split}
\end{equation}

Let $H \sim \Bin \left(|M_r|, q_{\frac{r}{\sigma}}(\ndist) \right)$. Using a simple coupling argument we get the bound:

\begin{equation}
    |M^*| \geq \sum_{e \in M_r} \indic{e \in \gaug} \geq H.
\end{equation}

Given our hypothesis we have that,
\begin{equation}
    \Pr{\babs{H-\Ex{H}} \geq \varepsilon \Ex{H}} \leq \frac{|M_r|q_{\frac{r}{\sigma}}(\ndist)(1-q_{\frac{r}{\sigma}}(\ndist))}{\varepsilon^2 |M_r|^2 q^2_{\frac{r}{\sigma}}(\ndist)} \leq \frac{1}{\varepsilon^2 |M_r|q_{\frac{r}{\sigma}}(\ndist)} \tolim{n} 0.
\end{equation}
It follows that for any $\varepsilon > 0$ we have that $H \geq (1-\varepsilon) |M_r| q_{\frac{r}{\sigma}}$ with high probability and so the result follows using \Cref{lem:gaug_lb}.
\end{proof}
We now show that our lower bound on $\Ex{|M_r|}$ from \Cref{lem:matching_exp} also holds with high probability:
\begin{lemma}
\label{lem:matching_hp}
    Let $X_1, \dots, X_n \iid \idist$ and $M_r$ be the largest matching in $G(\{X_1, \dots, X_n\}, r, d, \enorm)$ and assume that for each $n$ assumption \ref{cond:i_pos_low_dim} from \Cref{mod:low_dim} holds for $\idist$ and some $\enorm, R_d, \gamma$ and $\beta$, with the last two constants not depending on $n$. If $e^{-9 \beta d}\left(\frac{r}{R_d}\right)^d n^2 \tolim{n} +\infty$, then with high probability,
   \begin{equation}
    |M_r| \geq \frac{\gamma^2}{32} e^{-6\beta d} \left( \left(n^2 \left( \frac{r}{R_d} \right)^d\right) \wedge n \right).
\end{equation}
\end{lemma}
\begin{proof}

Letting once again,
\begin{equation}
    l = \frac{2 R_d e^{-\beta}}{(n-1)^{\frac{1}{d}}},
\end{equation}
assume first that $r \leq l$, and thus $r < R_d$. Let $N, B_d, p_{i,d}$, $U_{i,d}$ and $I_n$ be as in the proof of \Cref{lem:matching_exp}.
We will verify that $I_n$ concentrates around its expected value $\Ex{I_n}$, \textit{i.e.} that for any $\varepsilon > 0$,
\begin{equation}
(1-\varepsilon)\Ex{I_n} \leq I_n \leq (1+\varepsilon)\Ex{I_n},
\end{equation}
holds with high probability.
Using Chebyshev's inequality, it suffices to show that:
\begin{equation}
   \frac{\Ex{I^2_n}}{\Ex{I_n}^2} \tolim{n} 1. 
\end{equation}
Note that,
\begin{equation}
\label{eq:sum_1}
    \Ex{I^2_n} = \sum_{j=1}^{k_n} \sum_{i=1}^{k_n} \Pr{U_{i,d} = 2, U_{j,d} = 2}
\end{equation}
and 
\begin{equation}
\label{eq:sum_2}
    \Ex{I_n}^2 = \sum_{j=1}^{k_n} \sum_{i=1}^{k_n} \Pr{U_{i,d} = 2} \Pr{U_{j,d} = 2}.
\end{equation}
Also, using \eqref{eq:ball_prob_bound} and \eqref{eq:bound_kn},
\begin{equation}
    \sum_{i=1}^{k_n} \Pr{U_{i,d} = 2}^2 \leq \sum_{i=1}^{k_n} \Pr{U_{i,d} = 2} \leq \binom{n}{2} e^{2\beta d} \left(\frac{r}{8R_d}\right)^{d}.
\end{equation}
On the other hand, for $k_n \geq 2$, letting $q = \Pr{X_1 \in B_d}$,
\begin{equation}
\begin{split}
    \Ex{I_n}^2 &\geq \sum_{i=1}^{k_n} \sum_{j \in [k_n] \backslash \{ i\}} \Pr{U_{i,d} = 2}\Pr{U_{j,d} = 2} \\
    &\geq \frac{k^2_n}{2} q^4 \binom{n}{2}^2 \frac{e^{-4\beta d}}{16} \left(\frac{r}{2R_d}\right)^{4d} \\
    &\geq q^4 \binom{n}{2}^2 \frac{e^{-4\beta d}}{32} \left(\frac{r}{8R_d}\right)^{2d} \\
    &\gg \binom{n}{2} e^{2\beta d} \left(\frac{r}{8R_d}\right)^{d},
\end{split}
\end{equation}
since $e^{-9 \beta d}\left(\frac{r}{R_d}\right)^d n^2 \tolim{n} +\infty$ by assumption. Since $\Ex{I^2_n} \geq \Ex{I_n}^2$, it follows that both sums \eqref{eq:sum_1} and $\eqref{eq:sum_2}$ are dominated by the sums over terms with $i, j \in [k_n]$ with $i \neq j$. So to conclude that $\frac{\Ex{I^2_n}}{\Ex{I_d}^2} \tolim{n} 1$, it suffices to show that for any $\varepsilon > 0$ and $n$ large enough,
\begin{equation}
\label{eq:ratio_probs}
    \Pr{U_{i,d} = 2}\Pr{U_{j,d} = 2} \geq (1-\varepsilon) \Pr{U_{i,d} = 2, U_{j,d} = 2}, ~ \forall i \neq j \in [k_n].
\end{equation}
Now, given $\delta > 0$,
\begin{equation}
\begin{split}
\Pr{U_{i,d} = 2}\Pr{U_{j,d} = 2} &= \Ex{\binom{N}{2}p^2_{i,d}(1-p_{i,d})^{N-2}}\Ex{\binom{N}{2}p^2_{j,d}(1-p_{j,d})^{N-2}} \\
&\geq \Pr{|N - qn| < \delta qn}^2 \binom{(1-\delta)qn}{2}^2 p^2_{i,d} p^2_{j,d} (1-p_{i,d})^{(1+\delta)qn} (1-p_{j,d})^{(1+\delta)qn},
\end{split}    
\end{equation}
and,
\begin{equation}
\begin{split}
      \Pr{U_{i,d} = 2, U_{j,d} = 2} &= \Ex{\binom{N}{4}\binom{4}{2} p^2_{i,d} p^2_{j,d} (1-p_{i,d}-p_{j,d})^{N-4}} \\
      &\leq \Pr{|N - qn| < \delta qn} \frac{(1+\delta)^4 q^4 n^4}{4} p^2_{i,d} p^2_{j,d} (1-p_{i,d}-p_{j,d})^{(1-\delta)qn-4} \\
      &+ \Pr{|N - qn| \geq \delta qn}. 
\end{split}
\end{equation}
Note that,
\begin{equation}
\begin{split}
\label{eq:exp_conv}
    \log\left( \frac{(1-p_{i,d}-p_{j,d})^{(1-\delta)qn-4}}{(1-p_{i,d})^{(1+\delta)qn} (1-p_{j,d})^{(1+\delta)qn}} \right) 
    &= - qn \log \left( 1 + \frac{p_{i,d} p_{j,d}}{(1 - p_{i,d} - p_{j,d})} \right) \\
    &- (\delta q n + 4) \log \left(1-p_{i,d}-p_{j,d} \right) \\
    &-\delta q n \log\left( 1 - p_{i,d} \right) \\
    &-\delta q n \log\left( 1 - p_{j,d} \right).
\end{split}
\end{equation}
For sufficiently large $n$, using our assumption that $r \leq l$,
\begin{equation}
\begin{split}
\left|qn \log \left( 1 + \frac{p_{i,d} p_{j,d}}{(1 - p_{i,d} - p_{j,d})} \right)\right| &\leq 2 q n p_{i,d} p_{j,d} \\
&\leq 2 q n e^{2\beta d} \left( \frac{r}{2R_d}\right)^{2d}\\
&\leq  2 q n \frac{1}{(n-1)^2} \tolim{n} 0.
\end{split}
\end{equation}
Also, for sufficiently large $n$,
\begin{equation}
    |(\delta q n + 4) \log \left(1-p_{i,d}-p_{j,d} \right)| \leq 2 \delta q n \left|\log\left(1-\frac{2}{(n-1)}\right)\right| \tolim{n} 4 \delta q.
\end{equation}
Since,
\begin{equation}
    (\delta q n + 4) |\log \left(1-p_{i,d}-p_{j,d} \right)| \geq \delta q n |\log\left( 1 - p_{i,d} \right)| \vee \delta q n |\log\left( 1 - p_{j,d} \right)|,
\end{equation}
it follows that,
\begin{equation}
    \limsup_{n \to +\infty} \left|\log\left( \frac{(1-p_{i,d}-p_{j,d})^{(1-\delta)qn-4}}{(1-p_{i,d})^{(1+\delta)qn} (1-p_{j,d})^{(1+\delta)qn}} \right) \right| \leq 12 \delta q,
\end{equation}
thus taking $\delta > 0$ sufficiently small, for all sufficiently large $n$,
\begin{equation}
    \frac{(1-p_{i,d})^{(1+\delta)qn} (1-p_{j,d})^{(1+\delta)qn}}{(1-p_{i,d}-p_{j,d})^{(1-\delta)qn-4}} \geq \left(1-\frac{\varepsilon}{2}\right).
\end{equation}
From this \eqref{eq:ratio_probs} follows taking $n$ large and $\delta$ smaller if necessary to make the other factors of the ratio become closer to $1$. It follows that,
\begin{equation}
    \frac{\Ex{I^2_n}}{\Ex{I_n}^2} \tolim{n} 1,
\end{equation}
from which it follows that $I_n \geq \frac{\Ex{I_n}}{2}$ with high probability. Repeating the last part of the proof of \Cref{lem:matching_exp} for the case $r > l$, it follows that,
\begin{equation}
    |M_r| \geq \frac{\gamma^2}{32} e^{-d(3\beta+\log(8))} \left(n^2 \left( \frac{r}{R_d} \right)^d \wedge n \right)
\end{equation}
with high probability and thus the result follows.

\end{proof}
With the above results the proof of \Cref{thm:hp_lss_lb} follows readily:
\begin{proof}[Proof of \Cref{thm:hp_lss_lb}]
    The result follows by combining \Cref{lem:lss_gen_low_bound} and \Cref{lem:matching_hp} by taking $r = R_d\sigma$.
\end{proof}

\subsection{Information theoretic results}
\label{a:inf_theo_lb}

In this subsection we give the following positive probability analogue of \Cref{thm:lb_low_d}:

\begin{theorem}[Positive probability lower bound]
\label{thm:pos_pr_lb_low_d}
Assume that for each $n$, \Cref{mod:low_dim} in \eqref{eq:final_pos} holds for some $R_d, \enorm, \gamma$ and $\beta$, with the last two constants not depending on $n$. If $e^{-10 \beta d}\sigma^d n^2 \to +\infty$, then as $n \to +\infty$,
\begin{equation}
\label{eq:pos_pr_lb}
       \inf_{\hpi}\Pr{\hamdist{\hpi}{\pi^*} \geq \frac{\gamma^2}{128} e^{-7 \beta d} \left( (n^2 \sigma^d) \wedge n \right)} \geq \frac{(1-o(1))}{8},
\end{equation}
where the infimum above is taken over all estimators $\hpi$.
\end{theorem}
This result in particular completes the picture given by Theorem 3 in \cite{wang2022random} for $d = o(\log n)$ and generalizes their conclusions when $d$ is fixed for a wider class of distributions:

\begin{corollary}
\label{cor:hp_lbs}
    Given the assumptions of \Cref{thm:pos_pr_lb_low_d} and $d = o(\log n)$, for any estimator $\hpi$ and $\alpha \in [0,1)$, if
    \begin{equation}
        \gm = o(n^{\alpha}) \text{ with high probability},
    \end{equation}
    then,
    \begin{equation}
            \sigma = o\left( n^{\frac{\alpha - 2}{d}}\right).
    \end{equation}
    If we assume additionally that $d$ is fixed, the result also holds for $\alpha = 1$.
\end{corollary}
\begin{proof}
    $\gm = o(n^{\alpha})$ with high probability implies that for any $\varepsilon > 0$,
    \begin{equation}
        \Pr{\gm \leq \varepsilon n^{\alpha}} \tolim{n} 1.
    \end{equation}
    However, from our lower bound in \Cref{thm:pos_pr_lb_low_d},
    \begin{equation}
        \varepsilon n^{\alpha} \geq \frac{\gamma^2}{128} e^{-7\beta d} \left( (n^2 \sigma^d) \wedge n \right),
    \end{equation}
    for sufficiently large $n$. Since $e^{-7\beta d}n = n^{1-o(1)} > n^{\alpha}$ for sufficiently large $n$, this implies that,
    \begin{equation}
        \varepsilon n^{\alpha} \geq \frac{\gamma^2}{128} e^{-7\beta d} n^2 \sigma^d,
    \end{equation}
    and thus,
    \begin{equation}
        \sigma = o\left( n^{\frac{\alpha - 2}{d}}\right),
    \end{equation}
    as claimed. If we assume that $d$ is fixed and $\gm = o(n)$ with high probability, then for any $\varepsilon > 0$,
    \begin{equation}
        \Pr{\gm < \varepsilon \frac{\gamma^2}{128} e^{-7\beta d}n } \tolim{n} 1
    \end{equation}
    and in particular we must have:
    \begin{equation}
        \varepsilon n \geq n^2 \sigma^d,
    \end{equation}
    which implies $\sigma = o(n^{-\frac{1}{d}})$.
\end{proof}
In order to prove \Cref{thm:pos_pr_lb_low_d} we prove the following analogue of \Cref{lem:gen_inf_theo_lb}:

\begin{lemma}[Positive probability lower bound on $\gm$]
\label{lem:gen_pos_pr_lb}
Given the same setting as \Cref{lem:gen_inf_theo_lb}, if $|M_r|t_{\frac{r}{\sigma}}(\ndist) \geq 8$, then for any estimator $\hpi = \hpi(x_1, \dots, x_n, Y_1, \dots, Y_n) \in S_n$,
\begin{equation}
    \Pr{\gm \geq \frac{|M_r|t_{\frac{r}{\sigma}}(\ndist)}{4}} \geq \frac{1}{8},
\end{equation}
for $t_{s}(\ndist)$ as in \eqref{eq:t_def} and $M_r = M_r(\{x_1,\dots,x_n\})$ the largest matching in $G(\{x_1, \dots, x_n \}, r, d, \enorm)$.
\end{lemma}
\begin{proof}
    Assume once again for ease of notation that $M_r = \{\{ 2j-1,2j\} \}_{j=1}^m$ and let $\ndist(v, \sigma)$ denote the law of $v + \sigma \nnoise$, $\nnoise \sim \ndist$. Note that given $P$ and $Q$ two measures on $(\R^{d})^2$,
    \begin{equation}
        \tvnorm{P - Q} = \inf_{(W^{(0)},W^{(1)})} \Pr{ W^{(0)} \neq W^{(1)} },
    \end{equation}
    where the infimum is taken over all pairs $(W^{(0)},W^{(1)})$ such that $W^{(0)} \sim P$ and $W^{(1)} \sim Q$. Following \cite{villani2009optimal} Theorem 4.1 we have that for any pair $x_{2j-1}, x_{2j} \in \R^d$ there exists an optimal coupling $(W^{(0)},W^{(1)})$ in $(\R^{d})^2 \times (\R^d)^2$ with,
    \begin{equation}
        W^{(0)} \sim Q(x_{2j-1},\sigma) \otimes Q(x_{2j},\sigma) \text{ and } W^{(1)} \sim Q(x_{2j},\sigma) \otimes Q(x_{2j-1},\sigma)
    \end{equation}
    such that,
    \begin{equation}
       \Pr{W^{(0)} = W^{(1)}} = 1 - \tvnorm{Q(x_{2j-1},\sigma) \otimes Q(x_{2j},\sigma) - Q(x_{2j},\sigma) \otimes Q(x_{2j-1},\sigma)} =: p_j.
    \end{equation}
    Note that from \eqref{eq:tv_bound} we know that for all $j \in [m]$,
    \begin{equation}
    \label{eq:pj_lb}
        p_j \geq t_{\frac{r}{\sigma}}(\ndist).
    \end{equation}
    Let $\{(W^{(0)}_j,W^{(1)}_j)\}_{j=1}^m$ be independent optimal couplings of the respective collection of pairs of laws,
    \begin{equation}
        \{ \left(Q(x_{2j-1},\sigma) \otimes Q(x_{2j},\sigma), Q(x_{2j},\sigma) \otimes Q(x_{2j-1},\sigma) \right) \}_{j=1}^m.
    \end{equation}
   Let $C_j$ be the indicator of the event $W^{(0)}_j = W^{(1)}_j$. We define $b$ and $b'$ two random elements of $\{0,1\}^m$ by doing the following procedure independently for each $j \in [m]$:
    \begin{enumerate}[(i)]
        \item $b_j, b'_j \iid \text{Unif}(\{0, 1\})$ if $C_j = 1$,
        \item $b_j \sim \text{Unif}(\{0, 1\})$ and $b'_j = b_j$ if $C_j = 0$.
    \end{enumerate}
The key aspect of this construction is that although the joint pair $(b_j,b_j')$ depends on $C_j$, taken individually, both $b_j$ and $b'_j$ have the same marginal distribution and are each independent of $C_j$, and thus from $W_j^{(0)}, W_j^{(1)}$. Let $\nnoise_{1},\dots, \nnoise_n \iid \ndist$ be generated independently. Given $\omega \in \{0,1\}^m$ we define a sequence of final positions $Y^{(\omega)}_1, \dots, Y^{(\omega)}_n$ by:
    \begin{equation}
        (Y^{(\omega)}_{2j-1},Y^{(\omega)}_{2j}) = W^{(\omega_j)}_j, ~ j \in [m],
    \end{equation}
    and,
    \begin{equation}
        Y^{(\omega)}_i = x_i + \sigma \nnoise_i, ~ i = 2m + 1, \dots, n.
    \end{equation}
    Note that $Y^{(b)}_1, \dots, Y^{(b)}_n = Y^{(b')}_1, \dots, Y^{(b')}_n$ by construction. Moreover, conditioned on $b = \omega$, since $b$ is independent of the $W_j^{(0)}$'s and $W_j^{(1)}$'s,
    \begin{equation}
        Y^{(b)}_1, \dots, Y^{(b)}_n \sim \bigotimes_{i=1}^n \ndist\left( x_{\pi_{\omega}(i)}, \sigma \right),
    \end{equation}
    where $\pi_\omega = \prod_{j=1} ((2j-1)2j)^{\omega_j}$. By symmetry the same conclusion holds for $b'$. Since $b, b' \sim \mathrm{Unif}\left( \{ 0,1 \}^m \right)$ it follows that,
    \begin{equation}
        (b,Y^{(b)}_1, \dots, Y^{(b)}_n) \stackrel{(d)}{=} (b',Y^{(b')}_1, \dots, Y^{(b')}_n).
    \end{equation}
    Let $\pi^* \sim \mathrm{Unif}(S_n)$ be sampled independently. Note that conditioned on $b = \omega$ and $\pi^* = \pi$,
    \begin{equation}
        Y^{(b)}_{\pi^*(1)}, \dots, Y^{(b)}_{\pi^*(n)} \sim \bigotimes_{i=1}^n \ndist\left( x_{(\pi_{\omega} \circ \pi)(i)}, \sigma \right).
    \end{equation}
    Since $\pi_b \circ \pi^*, \pi_{b'} \circ \pi^* \sim \mathrm{Unif}(S_n)$, it follows that,
    \begin{equation}
    \begin{split}
        (\pi_{b} \circ \pi^*,Y^{(b)}_{\pi^*(1)}, \dots, Y^{(b)}_{\pi^*(n)}) &\stackrel{(d)}{=} (\pi_{b'} \circ \pi^*,Y^{(b')}_{\pi^*(1)}, \dots, Y^{(b')}_{\pi^*(n)}) \\
        &\stackrel{(d)}{=} (\pi^*,Y_1,\dots,Y_n),
    \end{split}
    \end{equation}
    where $Y_1,\dots,Y_n$ is defined by:
    \begin{equation}
        Y_i = x_{\pi^*(i)} + \sigma \nnoise_i, ~ i \in [n].
    \end{equation}
    In particular, since $\hpi(x_1, \dots, x_n, Y^{(b)}_{\pi^*(1)}, \dots, Y^{(b)}_{\pi^*(n)}) = \hpi(x_1, \dots, x_n, Y^{(b')}_{\pi^*(1)}, \dots, Y^{(b')}_{\pi^*(n)})$, setting,
    \begin{equation}
        \hpi_b = \hpi(x_1, \dots, x_n, Y^{(b)}_{\pi^*(1)}, \dots, Y^{(b)}_{\pi^*(n)}),
    \end{equation}
    for ease of notation, we have by the triangular inequality,
    \begin{equation}
    \label{eq:trig_ineq_dh}
    \begin{split}
        \hamdist{\hpi_b}{\pi_b \circ \pi^*} + \hamdist{\hpi_b}{\pi_{b'} \circ \pi^*} &\geq  \hamdist{\pi_b \circ \pi^*}{\pi_{b'} \circ \pi^*}\\
        &= 2\hamdist{b'}{b}.
    \end{split}
    \end{equation}
    Now, by a coupling argument using \eqref{eq:pj_lb} and the independence of $\{(b_j, b'_j) \}_{j=1}^m$, 
    \begin{equation}
        \hamdist{b'}{b} \geq H,
    \end{equation}
    where $H \sim \mathrm{Binom}(m, \frac{1}{2}t_{\frac{r}{\sigma}}(\ndist))$. Since $|M_r|t_{\frac{r}{\sigma}}(\ndist) \geq 8$ using Chernoff's inequality,
    \begin{equation}
         \Pr{H \geq \frac{t_{\frac{r}{\sigma}}(\ndist)|M_r|}{4}} \geq \Pr{H > \frac{\Ex{H}}{2}} \geq 1 - e^{-\frac{\Ex{H}}{8}}\geq \frac{1}{4}.
    \end{equation}
    We thus have that using \eqref{eq:trig_ineq_dh}:
    \begin{equation}
    \begin{split}
        \Pr{\hamdist{\hpi_b}{\pi_b \circ \pi^*} \geq \frac{t_{\frac{r}{\sigma}}(\ndist)|M_r|}{4} } + \Pr{\hamdist{\hpi_b}{\pi_{b'} \circ \pi^*} \geq \frac{t_{\frac{r}{\sigma}}(\ndist)|M_r|}{4} } &\geq \Pr{\hamdist{b'}{b} \geq \frac{t_{\frac{r}{\sigma}}(\ndist)|M_r|}{4} } \\
        &\geq \Pr{H \geq \frac{t_{\frac{r}{\sigma}}(\ndist)|M_r|}{4}} \\
        &\geq \frac{1}{4}.
    \end{split}
    \end{equation}
    Note that,
    \begin{equation}
    \begin{split}
        \hamdist{\hpi(x_1, \dots, x_n, Y^{(b)}_1, \dots, Y^{(b)}_n)}{\pi_b \circ \pi^*} &\stackrel{(d)}{=} \hamdist{\hpi(x_1, \dots, x_n, Y_1, \dots, Y_n)}{\pi^*} \\ 
        &\stackrel{(d)}{=} \hamdist{\hpi(x_1, \dots, x_n, Y^{(b')}_1, \dots, Y^{(b')}_n)}{\pi_{b'} \circ \pi^*}
    \end{split}
    \end{equation}
    and so,
    \begin{equation}
        \Pr{\hamdist{\hpi(x_1, \dots, x_n, Y_1, \dots, Y_n)}{\pi^*} \geq \frac{t_{\frac{r}{\sigma}}(\ndist)|M_r|}{4} } \geq \frac{1}{8},
    \end{equation}
    as desired.
    
\end{proof}
We are now in position to prove \Cref{thm:pos_pr_lb_low_d}:
\begin{proof}[Proof of \Cref{thm:pos_pr_lb_low_d}]
    Just combine \Cref{lem:gen_pos_pr_lb} and \Cref{lem:matching_hp} taking $r = R_d \sigma$.
\end{proof}

\pagebreak

\section{Recovering lower bounds for the logarithmic regime in the Gaussian model}
\label{a:log_reg}

In this section we show that if \Cref{conj:size_largest_matching} holds, then we can use \Cref{lem:lss_gen_low_bound} to recover the lower bound from \cite{geo_match_2022} on the logarithmic regime for some choices of $a > 0$ and $\sigma^2 > 0$.

We assume the Gaussian model from \cite{geo_match_2022}, \ie $\idist = \mathcal{N}(0,I_d)$ and $\ndist = \mathcal{N}(0,\sigma^2 I_d)$ in \eqref{eq:final_pos}, and the logarithmic regime \ie $d = a\log n$. Using the framework from \Cref{lem:lss_gen_low_bound}, for the Gaussian model we have:
\begin{equation}
    q_{\frac{r}{\sigma}}(\ndist) = \parent{1-\Phi\left(\frac{r}{\sqrt{2}\sigma}\right)},
\end{equation}
where $\Phi$ is the cumulative distribution function of a standard Gaussian random variable.
Let,
\begin{equation}
    I(t) = \frac{1}{2}\left( t - 1 - \log\left( t \right) \right) \text{ for } t > 0
\end{equation}
be the large deviation rate function of a chi-squared random variable with $1$ degree of freedom. For a fixed $\sigma^2 > 0$ consider $r^2 = \gamma d$ for some $\gamma \in (0,2)$. Letting $G(n,r,d) = G(n,r,d,\ltwo{\,\cdot\,})$ be the random geometric graph induced by joining initial positions within Euclidean distance $r$, we have that:

\begin{align}
    \log\left(\Ex{E(G(n,r,d))} \left( 1 - \Phi \left(\frac{r}{\sqrt{2} \sigma }\right) \right)\right)
    &= \log \left( \binom{n}{2} \right) + \log \left( \Pr{ \| X_1 - X_2 \|_2^2 \leq \gamma d } \right) + \log \left(1 - \Phi \left(\frac{\sqrt{\gamma d}}{\sqrt{2 }\sigma}\right) \right).
\end{align}
Dividing all terms by $\log n$, taking the limit as $n \to \infty$, and using that $d = a \log n$,

\begin{align}
    \lim_{n \to +\infty} \frac{\log\left(\Ex{E(G(n,r,d))} \left( 1 - \Phi \left(\frac{r}{\sqrt{2} \sigma }\right) \right)\right)}{\log n} &= 2 - a I \left(\frac{\gamma}{2} \right) - \frac{a \gamma}{4 \sigma^2} \\
    &= 2 - \frac{a}{2} \left( 2 I \left( \frac{\gamma}{2} \right) + \frac{\gamma}{2 \sigma^2} \right) \\
    & = 2 - \frac{a}{2} \left( \frac{\gamma}{2} \left( 1 + \frac{1}{\sigma^2} \right) - 1 - \log\left( \frac{\gamma}{2} \right) \right).
\end{align}
The last expression is maximized by taking $\gamma^* = 2\left( 1 + \frac{1}{\sigma^2} \right)^{-1}$ which gives a value of $2 - \frac{a}{2} \log \left( 1 + \frac{1}{\sigma^2} \right)$, precisely the rate established in \cite{geo_match_2022}. Let $r^* = \sqrt{\gamma^* d}$. If we assume \Cref{conj:size_largest_matching} then $|M_{r^*}| = \Theta\parent{\Ex{E(G(n,r^*,d))} \wedge n}$ with high probability, from which it follows that:
\begin{equation}
    \frac{\log\parent{|M_{r^*}|}}{\log\parent{\Ex{E(G(n,r^*,d))} \wedge n}} \pc{n} 1,
\end{equation}
since $\Ex{E(G(n,r^*,d))} \to +\infty$.
In particular, if $1 \ll \Ex{E(G(n,r^*,d))} \ll n$, then by \Cref{lem:lss_gen_low_bound}, we have with high probability that,
\begin{equation}
\begin{split}
        \frac{\log\parent{\lssm \vee 1} }{\log\parent{n}} \geq \frac{\log\parent{|M_{r^*}|q_{\frac{r^*}{\sigma}}}-\log(2)}{\log n}.
\end{split}
\end{equation}
Since the last expression converges to $2 - \frac{a}{2} \log \left( 1 + \frac{1}{\sigma^2} \right)$ in probability we have that,
\begin{equation}
\label{eq:log_lb}
    \lssm \vee 1 \geq n^{\alpha + o_p(1)},
\end{equation}
where $\alpha = 2 - \frac{a}{2} \log \left( 1 + \frac{1}{\sigma^2} \right)$.

Unfortunately the additional constraint that $1 \ll \Ex{E(G(n,r^*,d))} \ll n$ implies that $2 - aI\parent{\frac{\gamma^*}{2}} \in (0,1)$ and so in \eqref{eq:log_lb} we have an additional constraint on $a > 0$ and $\sigma^2 > 0$. Moreover, similar to \cite{geo_match_2022} we are unable to establish a lower bound of order $\Omega(n)$ with this strategy, as it also only deals with augmenting cycles of size $2$.

\end{appendices}

\cleardoublepage
\addcontentsline{toc}{subsection}{\refname} 
\printbibliography

\pagebreak

\end{document}

%% file: packages.tex
\usepackage{fullpage}  

\usepackage[english]{babel}  

\usepackage{amsmath}  
\usepackage{amsthm}  
\usepackage{amssymb}  
\usepackage{amsfonts} 
\usepackage{latexsym} 
\usepackage{bbm}  
\usepackage{bm}  

\usepackage{booktabs}

\usepackage[shortlabels]{enumitem}  

\usepackage{graphicx}  
\usepackage[dvipsnames]{xcolor}  

\graphicspath{{figures/}}  
\graphicspath{{../figures/}}  
\usepackage{float}
\usepackage{neuralnetwork} 

\usepackage{breakcites}  
\usepackage{url}  

\usepackage{mathtools}  
\usepackage{mathrsfs}

\usepackage[utf8]{inputenc}

\usepackage{listings}

\usepackage{comment}  

\usepackage[utf8]{inputenc}

\usepackage{graphicx}
\usepackage{setspace}
\usepackage{mathtools}
\usepackage[top=1.0in,left=1.0in,bottom=1in,right=1.0in]{geometry}
\usepackage{accents}
\usepackage[title,titletoc]{appendix}
\usepackage{caption}
\usepackage{subcaption}
\usepackage{authblk}
\usepackage[maxnames=50, giveninits=true,backend=biber,
style=alphabetic]{biblatex}
\usepackage{csquotes}

\usepackage{xr-hyper}
\usepackage{hyperref}
\usepackage{import}

\hypersetup{colorlinks, citecolor=black, filecolor=black,
    		linkcolor=black, urlcolor=black}

\usepackage{cleveref}

%% file: definitions.tex
\newcommand{\ie}{{\it i.e. }}

\renewcommand{\P}{\mathbb{P}}  
\renewcommand{\Pr}[1]{\mathbb{P}\left( #1 \right)}
\newcommand{\Ex}[1]{\mathbb{E}\left[ #1 \right]}  
\newcommand{\Exs}[2]{\mathbb{E}_{#1}\left[ #2 \right]}  

\newcommand{\iid}{\stackrel{\text{iid}}{\sim}}

\newcommand{\Bin}{\textnormal{Bin}}

\newcommand{\indic}[1]{\mathbf{1}_{[#1]}}

\newcommand{\R}{\mathbb{R}}  
\newcommand{\Z}{\mathbb{Z}}  
\newcommand{\N}{\mathbb{N}}  

\newcommand{\rvseq}[1]{\left(#1_n \right)_{n \in \N}}

\newcommand{\tolim}[1]{\underset{#1 \to \infty}{\longrightarrow}}

\newcommand{\dc}[1]{\overset{(d)}{\underset{#1 \to \infty}{\longrightarrow}}}
\newcommand{\pc}[1]{\overset{(\P)}{\underset{#1 \to \infty}{\longrightarrow}}}
\newcommand{\asc}[1]{\overset{\text{a.s.}}{\underset{#1 \to \infty}{\longrightarrow}}}

\newcommand{\tr}[1]{\mathop\mathrm{tr}\parent{#1}}
\newcommand{\diago}[1]{\mathop\mathrm{diag}\parent{#1}}

\DeclareMathOperator*{\esssup}{ess\,sup}
\DeclareMathOperator*{\essinf}{ess\,inf}

\newcommand{\inner}[2]{\langle{#1},{#2}\rangle}  
\newcommand{\binner}[2]{\left\langle{#1},{#2}\right\rangle}  
\def\<{\left\langle}  
\def\>{\right\rangle}

\newcommand{\norm}[1]{\left\| {#1} \right\|}  
\newcommand{\tvnorm}[1]{\left\|{#1}\right\|_{TV}} 
\newcommand{\snorm}[1]{\left\|{#1}\right\|_{\psi_2}}  
\newcommand{\hamdist}[2]{d_H\left({#1},{#2}\right)}  
\newcommand{\lone}[1]{\norm{#1}_1}  
\newcommand{\ltwo}[1]{\norm{#1}_2}  
\newcommand{\linf}[1]{\norm{#1}_\infty}  
\newcommand{\fnorm}[1]{\ltwo{#1}} 
\newcommand{\opnorm}[1]{\norm{#1}_{op}} 
\newcommand{\enorm}{\|\,\cdot\,\|} 

\newcommand{\expo}[1]{\exp\left(#1\right)} 
\newcommand{\set}[1]{\left\{ #1 \right\}} 
\newcommand{\parent}[1]{\left( #1 \right)} 
\newcommand{\babs}[1]{\bigl\lvert #1 \bigr\rvert} 
\newcommand{\ubar}[1]{\underaccent{\bar}{#1}} 

\DeclareFontFamily{U}{skulls}{}
\DeclareFontShape{U}{skulls}{m}{n}{ <-> skull }{}

\theoremstyle{plain}  
	\newtheorem{theorem}{Theorem}
	\newtheorem{proposition}[theorem]{Proposition}
	\newtheorem{lemma}[theorem]{Lemma}
    
	\newtheorem{corollary}[theorem]{Corollary}
	
 \newtheorem{conjecture}[theorem]{Conjecture}

\theoremstyle{definition}

    \newtheorem{modelinner}{Model}
    \newenvironment{model}[1]{%
  \renewcommand\themodelinner{#1}%
  \modelinner
}{\endmodelinner}
	
	\newtheorem{exampleEnv}[theorem]{Example}
			{\pushQED{\qed}\exampleEnv}
			{\popQED\endexampleEnv}
		\newenvironment{example*}{\exampleEnv}{\endexampleEnv}
	\newtheorem{remarkEnv}[theorem]{Remark}
		\newenvironment{remark}  
			{\remarkEnv}
			{\endremarkEnv}
		\newenvironment{remark*}{\remarkEnv}{\endremarkEnv}

  	\newtheorem{directionEnv}[theorem]{Direction}
			{\directionEnv}
			{\enddirectionEnv}
		\newenvironment{direction*}{\directionEnv}{\enddirectionEnv}
	
	\newtheorem*{solutionEnv}{Solução}
		\ifdefined\showsolution
			\newenvironment{solution}
				{\pushQED{\qed}\solutionEnv}
				{\popQED\endsolutionEnv}
			\newenvironment{solution*}{\solutionEnv}{\endsolutionEnv}
		\else  
			
			\excludecomment{solution}
			\newenvironment{solution*}{}{}
			\excludecomment{solution*}
		\fi

	\ifdefined\showcomment 
		\newtheorem*{comment}{\normalfont\emph{Comentário}}
	\else  
		\newenvironment{comment}{}{}
		\excludecomment{comment}
	\fi

\numberwithin{theorem}{section}

\def\lc{\left\lceil}   
\def\rc{\right\rceil}